\documentclass[reqno]{amsart}
\usepackage{comment}
\usepackage[margin=1in]{geometry} 
\usepackage{graphicx,mathtools,tikz,hyperref}
\usepackage{amsthm}
\usepackage{cite}
\usepackage{soul}
\usepackage[utf8]{inputenc}
\usepackage[english]{babel}

\usepackage{empheq}
\usepackage{enumitem}
\usepackage{theoremref}

\usepackage{pgfplots}
\usetikzlibrary{decorations.pathmorphing, patterns}
\usetikzlibrary{decorations.pathreplacing,angles,quotes}

\usetikzlibrary{positioning}
\newcommand{\real}{\mathbb{R}}

\numberwithin{equation}{section}
\theoremstyle{plain}
\newtheorem{theorem}{Theorem}[section]
\newtheorem{proposition}[theorem]{Proposition}

\newtheorem{lemma}[theorem]{Lemma}
\theoremstyle{remark}
\newtheorem{remark}[theorem]{Remark}
\theoremstyle{definition}

\begin{document}
\date{\today}

\title{Broadening global families of anti-plane shear equilibria}
\author[T. Hogancamp]{Thomas Hogancamp}
\address{Department of Mathematics, University of Missouri, Columbia, MO 65211}
\email{thvf6@mail.missouri.edu}
\begin{abstract}
    We develop a global bifurcation theory for two classes of nonlinear elastic materials. It is supposed that they are subjected to anti-plane shear deformation and occupy an infinite cylinder in the reference configuration. Curves of solutions to the corresponding elastostatic problem are constructed using analytic global bifurcation theory. The curve associated with first class is shown to exhibit broadening behavior, while for the second we find that the governing equation undergoes a loss ellipticity in the limit. A sequence of solutions undergoes broadening when their effective supports grow without bound. This phenomena has received considerable attention in the context of solitary water waves; it has been predicted numerically, yet it remains to be proven rigorously. The breakdown of ellipticity is related to cracks and instability making it an important aspect of the theory of failure mechanics.    
\end{abstract}
\maketitle

\begin{section}{Introduction}
Our primary goal is to rigorously prove the existence of interesting families of equilibria far from the reference configuration in the context of nonlinear elastostatics. Global bifurcation theorems for nonlinear elasticity have been established in, for example,  \cite{Healey_Simpson}, \cite{healey1997unbounded} and \cite{Healey_2018}. Due to the generality of the systems these authors consider, one must often accept that several alternatives may hold along the resulting global continuum. Efforts to develop global theories with complete characterizations have been limited. As one might expect, this requires some restrictions on the material, domain, and deformation. To this end, we focus on materials whose reference configuration is an unbounded cylinder that are in a state of anti-plane shear displacement and subjected to a parameter dependent body force. A material is said to be in a state of \textit{anti-plane shear} provided the deformation is of the form 
\begin{equation}\label{anti-plane_deform}
    \text{id} + u(x,y)e_{3}.
\end{equation} 
The relative simplicity of \eqref{anti-plane_deform} makes the study of such deformations an important pilot problem. For example, it is in this context that Saint-Venant's principle for nonlinear elasticity was first probed (see \cite[Section~6]{Horgan_antiplane}). Our investigation into broadening and ellipticity breakdown are in the same spirit. 

As with many global theories, the natural first step is to construct a perturbative family. We begin by further developing the local bifurcation theory established in \cite[Section 3]{chen2019center}. With this in hand, we employ the analytic global bifurcation theory of \cite{chen2018existence} to obtain a branch of solutions to the corresponding static equilibria problem that abides by a series of alternatives, which are in the same spirit as the ones mentioned above, that hold for a large class of materials. Sharp results are obtained by imposing reasonable assumptions on the material and body force; one set of conditions ensures broadening, while another leads to a loss of ellipticity. We make extensive use of monotonicity properties of solutions, elliptic type estimates, and a conserved quantity of the system in order to develop this global theory.    

\textit{Broadening} is characterized by the existence of a sequence of solutions to the relevant static equilibria problem, for which each element is uniformly bounded in an appropriate H\"{o}lder norm and decays to zero in the unbounded direction of the cylinder, yet does not admit a uniformly convergent subsequence. Such a sequence would fail to be uniformly spatially localized. It is in this sense that the displacements become broad. We also mention that after appropriately translating each element of such a sequence that one obtains a front type solution in the limit; see Figure~\ref{broadeningpic}. Broadening has been a topic of interest in the study of interfacial solitary water waves since at least the 1980s. Amick and Turner developed a global bifurcation theory which includes broadening as a possible alternative in \cite{amick_turner}, while Turner and Vanden-Broeck  predicted the phenomena numerically in \cite{turner1988broadening}. Whether broadening does indeed occur for such waves is still an open question. To the best of our knowledge, we are the first to rigorously construct a family of solutions that exhibit broadening in the PDE context. 

A loss of ellipticity occurs when the governing equations change type. This is possible for some materials subjected to subjected to deformations with sufficiently large gradients. Knowles explores the relationship between ellipticity and crack formation for nonlinear elastics in \cite{knowles1977finite} and again with Sternberg in \cite{knowles1981anti}. We construct families of solutions whose maximum deformation gradient must limit to the critical state where ellipticity breaks down. As far as we know, no previous global construction has been shown to exhibits such behavior. 

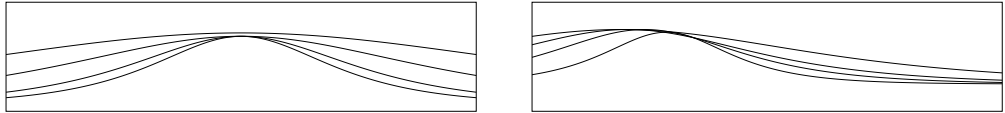
\begin{figure}
\centering
    \begin{tikzpicture}[scale=0.6]
      \begin{axis}[width=12cm, height =4cm, xmin=-300,xmax=300, ymin=-.1, ymax =1.5, yticklabels = \empty , xticklabels= \empty, ytick=\empty, xtick=\empty] 
    \addplot[domain =-300:300, samples =300] {1/(cosh(0.01*x))}; 
    \addplot[domain= -300: 300, samples=300] {1/cosh(0.008*x))}; 
    \addplot[domain = -300: 300, samples =300] {1.001/cosh(0.005*x))}; 
    \addplot[domain = -300: 300, samples =300] {1.05/cosh(0.003*x))}; 
     \end{axis}
    \end{tikzpicture}
    \hspace{0.5cm}
\begin{tikzpicture}[scale=0.6]
  \begin{axis}[width=12cm, height =4cm, xmin=-130,xmax=150, ymin=-0.5, ymax =1.5, yticklabels = \empty , xticklabels= \empty, ytick=\empty, xtick=\empty, 
  ] 
  \addplot[domain=-150:150, samples =300] {0.95/cosh(0.03*((x+50)))};
  \addplot[domain=-150:150, samples =300] {1/cosh(0.02*((x+63)))};
  \addplot[domain=-150:150, samples =300] {1/cosh(0.01*((x+77)))};
  \addplot[domain= -150:150, samples =300] {1/cosh(0.015*((x+73)))};
  \end{axis}
\end{tikzpicture}
    \caption{Left: Broadening along the center line $y=0$. Right: Shifted functions converging to a front along the center line $y=0$.}
    \label{broadeningpic}
\end{figure}

\begin{subsection}{The problem}
Consider a homogeneous, isotropic, incompressible, hyperelastic material occupying the region $\mathcal{D}:=\Omega \times \real$, where $\Omega = \real\times (-\frac{\pi}{2},\frac{\pi}{2})$. Let $u$ be the unknown displacement as in \eqref{anti-plane_deform}. Assume that $u =0$ on $\partial \Omega$, which may be interpreted physically as a clamped boundary condition. The structure of the static equilibrium equations is largely determined by the strain energy density function $W$. It is well known that for these materials $W=W(I_{1}, I_{2},I_{3})$, where $I_{1,2,3}$ are the principal invariants of the Cauchy--Green tensor, and that $I_{1}=I_{2}=3+|\nabla u|^{2}$ for anti-plane shear deformations (see \cite[Section 4]{Horgan_antiplane}). Also, note that incompressibility implies $I_{3}=1$.  Let us consider generalized neo-Hookean materials, in which case $W$ depends on $I_{1}$ alone. We write
\begin{gather*}
    W(I_{1},I_{2}, I_{3}) = : \overline{W}(I_{1})
\end{gather*}
to simplify the notation. Finally, let 
\begin{gather*}
    \mathcal{W}(q) :=  \overline{W}(3+q),   
\end{gather*}
where $\frac{1}{2}\mathcal{W}(q)$ is the so called \textit{modulus of shear} at amount of shear $q$ \cite{horgan1981effect}. 
 
We also suppose, following for example \cite{Healey_Simpson}, \cite{Healey_2018}, and \cite{chen2019center}, that a parameter dependent live load acts on the material. Let $b=b(z,\lambda)$ denote the associated force density. Both $\mathcal{W}$ and $b$ are required to be analytic in their arguments. Near the reference configuration, it is assumed that they have the expansions 
\begin{subequations}
\label{analytic_assumptions}
\begin{align}
    \mathcal{W}(q) = q+c_{1}q^2+c_{2}q^{3}+O(|q|^{4})& \label{w}\\
    b(z,\lambda) = (\lambda -1)z+b_{1}z^{3}+O(|z|^{5})& \label{b}
\end{align}
\end{subequations}
where $z$ is displacement and $\lambda$ the parameter of the live load. We consider the case in which $b_{1} \leq 0$, $c_{1}<0$, and $b$ is odd in $z$. The more general assumption that $b$ is odd in $z$ and $b_{1}+2c_{1}<0$ is shown to be a requirement for the existence of spatially localized solutions near the reference configuration in \cite{chen2019center}.

In general, the equations that describe anti-plane shear are an over-determined system consisting of three equations and two unknowns. Knowles gives necessary conditions for non-trivial states of anti-plane shear in the absence of body forces in \cite{Knowles_incompressible}. As he points out, generalized neo-Hookean materials always satisfy these conditions, and hence the governing equations are reduced to a single scalar PDE:
\begin{empheq}[left=\empheqlbrace]{align}
\label{cub}
\begin{split}
        \nabla \cdot (\mathcal{W}'(|\nabla u|^{2})\nabla u)-b(u,\lambda)&=0 \qquad \text{in} \; \Omega\\  
         u&=0 \qquad \text{on} \; \partial \Omega. 
\end{split}
\end{empheq}
Naturally, the structure of both $\mathcal{W}$ and $b$ will have a great effect on the qualitative features of the equilibria.  

The first elastic model we consider, which we refer to as Model I, is equation \eqref{cub} with a corresponding $\mathcal{W}$ satisfying 
\begin{subequations} \label{w_props_1}
\begin{align}
    \mathcal{W}'(q)+2q\mathcal{W}''(q) >\xi_{1}>0, \; \; q\geq 0&  \label{Wcond}\\
    q+c_{1}q^{2}+c_{2}q^{3} \leq \mathcal{W}(q), \; \; q \geq 0.& \label{wgrowth2}
\end{align}
\end{subequations}
Note that \eqref{Wcond} and \eqref{w} force the relation $c_{1}^{2} < \frac{5}{3} c_{2}$. Condition \eqref{Wcond} ensures that \eqref{cub} is uniformly elliptic regardless of the magnitude of the deformation gradient. The combined properties \eqref{w_props_1} are used to establish important a priori estimates. There is no universal choice for the growth conditions of $\mathcal{W}$, but we mention that polynomial models are common in the applied literature. In particular, the reduced polynomial model for incompressible materials has strain energy density given explicitly by:
\begin{gather}\label{yeoh}
\mathcal{W}(q)=\sum_{i=1}^{n}C_{i}q^{i}.
\end{gather}
After normalization so that $C_{1}=1$, we find that for a large class of coefficients \eqref{yeoh} will satisfy the assumptions of Model I. When $n=3$, \eqref{yeoh} becomes the widely used Yeoh model \cite{yeoh1993some}. The parameters of the Yeoh model may be chosen so that $C_{2}<0$, which is one of our assumptions, in order to capture some of the experimental properties of rubber \cite{yeoh1993some}. Our final assumption for Model I is that the the live load $b$ satisfies the following conditions: 
\begin{subequations}\label{both_b}
\begin{align}
    (\lambda -1)z +b_{1}z^{3}\leq b(z,\lambda), \;\; \text{for} \; \; z\geq 0 & \label{bcond} \\ 
    -b_{z}(0,\lambda) < 1, \; \; \text{for} \; \; \lambda>0&. \label{bz}
\end{align}
The condition \eqref{bcond} is used to help obtain a priori estimates, and \eqref{bz} is a requirement of the local theory for homoclinic solutions in \cite{chen2019center}.
\end{subequations}

We also consider a second model, Model II, which is again governed by equation \eqref{cub}, but where $\mathcal{W}$ satisfies 
\begin{subequations}\label{wcond_2}
\begin{align}
        \mathcal{W}'(q)+2q\mathcal{W}''(q)>0 \; \text{for all} \; q \in [0,q_{1})& \label{wnondegen}\\
        \mathcal{W}'(q)+2q\mathcal{W}''(q) \to 0 \;\; \text{as} \;\; q \to q_{1}^{-}& \label{wdegen} \\
        q\mathcal{W}'(q)-\mathcal{W}(q) < 0, \;\text{for} \; \; q>0. &\label{wdegendamp}
\end{align}
\end{subequations}
Here, \eqref{wnondegen} and \eqref{wdegen} mean simply that \eqref{cub} is elliptic so long as $|\nabla u|^{2}<q_{1}$, and that \eqref{cub} loses ellipticity as $|\nabla u|^{2} \to q_{1}^{-}$. Furthermore, in this case, we suppose that $b$ is concave in $z$ and satisfies \eqref{both_b}. Condition \eqref{wdegendamp} along with the concavity of $b$ will be used to help rule out broadening for Model II. 

Let us now consider a basic instantiation of Model II. The functions
\begin{gather}
\mathcal{W}(q) = q+c_{1}q^{2}, \;\; b(\lambda, z) = (\lambda - 1)z , 
\end{gather}
with $c_{1}<0$, correspond to a ``softening" elastic material undergoing simple harmonic forcing. The criteria for a softening material in the present notation is simply $\mathcal{W}''(q)<0$, for all $q>0$. See \cite{horgan1981effect} for more details on hardening or softening materials subject to anti-plane shear. 
\end{subsection}
\begin{subsection}{Main results}
The primary contribution of this paper is the construction of global families of solutions to \eqref{cub} that either broaden or lose ellipticity. 
\begin{theorem}[Model I]
\thlabel{thm1}
There is a curve $\mathcal{C}^{I}$ of solutions to \eqref{cub}, under the assumptions of Model I, admitting the $C^{0}$ parameterization 
\begin{gather*}
\mathcal{C}^{I} = \{ (u(s), \lambda(s)) : 0<s<\infty \} \subset C_{\textup{b}}^{3+\alpha}(\overline{\Omega})\times(0,\infty)
\end{gather*}
and $(u(s),\lambda(s))\to (0,0)$ as $s\to 0^{+}$. Moreover, we have that $\mathcal{C}^{I}$ satisfies the following:  
\begin{enumerate}[label=(\alph*),font=\upshape]
\item\textup{(Symmetry and monotonicity)} Each $(u(s),\lambda(s)) \in \mathcal{C}^{I}$ is monotone in the sense that \begin{gather*}
    \partial_{x}u(s) > 0 \; \text{for} \;x>0 \\ 
    \partial_{y}u(s) <0 \; \text{for} \; y>0, 
\end{gather*}
and $u(s)$ is even in both $x$ and $y$. \label{sm}
\item \textup{(Analyticity)} The curve $\mathcal{C}^{I}$ is locally real analytic. \label{analyticity}
\item \textup{(Bounds on $\lambda$)} \label{th1boundslambda} There exists some $0<\lambda_{1}^{-}<\lambda_{1}^{+}<\infty$ for which $s\gg 1$ implies 
\begin{gather*}
\lambda_{1}^{-}<\lambda(s) <\lambda_{1}^{+}.
\end{gather*} 
\item\textup{(Bounds on displacement)} There exists $C=C(c_{1},c_{2},b_{1})>0$ for which 
\begin{gather*}
    \sup_{s\geq 0}|u(s)|_{3+\alpha} \leq C.
\end{gather*} \label{bounds_displacement}
\item \textup{(Broadening)} There is a sequence $\{ (u_{n}, \lambda_{n})\} \subset \mathcal{C}^{I}$, and a sequence $\{x_{n}\}$, with $x_{n} \to \infty$, such that 
\begin{gather*}
     u_{n}(\cdot + x_{n}, \cdot) \xrightarrow{C^{3}_{\textup{loc}}(\overline{\Omega})} \tilde{u} \in C_{\textup{b}}^{3+\alpha}(\overline{\Omega}),
\end{gather*}
where $\tilde{u}$ is a solution to \eqref{cub}, $\tilde{u} \not\equiv 0$, and $\partial_{x}\tilde{u} \leq 0$. \label{broadening}
\end{enumerate}
\end{theorem}
We call $\tilde{u}$ in the broadening alternative a front since it has distinct limiting states as $x\to \infty$ and $x\to -\infty$. In fact, we will see that $\tilde{u}$ must decay to $0$ as $x \to \infty$ but limit to a non-trivial $x$-independent solutions of \eqref{cub} as $x \to -\infty$. Fronts are of great interest in, for example, the study of reaction-diffusion systems, hydrodynamics, and mathematical biology. In \cite{chen2020global}, the authors consider a problem similar to the one posed here. However, they assume that  $b_{1}+2c_{1}>0$ and are able to construct a global family of front-type solutions whose displacement grow arbitrarily large. Note that our assumption $b_{1},c_{1}<0$ is not quite complementary to the one above; however, when taken together these conditions exhaust the global theory for a wide range the parameters.     

Our second main result concerns Model II. As in \thref{thm1}, we establish monotonicity, local analyticity, and an upper bound on $\lambda$ along the global curve. However, in this case the strictly positive lower bound on $\lambda$ and the upper bound on the displacement are lost. Furthermore, the loss of ellipticity, which is the distinctive feature of \thref{thm2}, is an impossibility under the assumptions of Model I.  
\begin{theorem}[Model II]
\thlabel{thm2}
There is a curve $\mathcal{C}^{II}$ of solutions to \eqref{cub}, under the assumptions of Model II, with $C^{0}$ parameterization
\begin{gather*}
    \mathcal{C}^{II}=\{ (u(s), \lambda(s)) : 0<s<\infty \} \subset C_{\textup{b}}^{3+\alpha}(\overline{\Omega})\times(0,\infty)
\end{gather*}
and $(u(s),\lambda(s))\to (0,0)$ as $s\to 0^{+}$. Moreover, we have that $\mathcal{C}^{II}$ satisfies the following:  
\begin{enumerate}[label=(\alph*), font=\upshape]
\item\textup{(Symmetry and monotonicity)} Each $(u(s),\lambda(s)) \in \mathcal{C}^{I}$ is such that $u(s)$ is monotone in the sense that \begin{gather*}
    \partial_{x}u(s) > 0 \; \text{for} \;x>0 \\ 
    \partial_{y}u(s) <0 \; \text{for} \; y>0, 
\end{gather*}
and $u(s)$ is even in both $x$ and $y$. \label{sm2}
\item \textup{(Analyticity)} The curve $\mathcal{C}^{II}$ is locally real analytic. \label{analyticity2}
\item \textup{(Bounds on $\lambda$)} \label{th2boundslambda} There exists some $0<\lambda_{2}^{+}<\infty$ for which $s\gg 1$ implies 
\begin{gather*}
0<\lambda(s) <\lambda_{2}^{+}
\end{gather*} 
\item \textup{(Loss of ellipticity)} Following $\mathcal{C}^{II}$ to its extreme, the system loses ellipticity in that 
\begin{gather}\label{elliptic_loss_quant}
     \lim_{s\to \infty}\inf_{\overline{\Omega}}\big(\mathcal{W}'(q)+2q\mathcal{W}''(q)\big)\big\vert_{q=|\nabla u(s)|^{2}} = 0
\end{gather} \label{loss_ellipt}
\end{enumerate}

\end{theorem}
We note that working under only the assumptions of \eqref{analytic_assumptions}, \eqref{bz}, and $b_{1},c_{1}<0$, our methods would show the existence of a global curve of solutions that satisfy \ref{sm2} and \ref{analyticity2} as above. Moreover, the condition $0< \lambda(s)$ is retained. The conditions \eqref{wgrowth2} and \eqref{bcond} are only used once (in Section~\ref{UR}). They help ensure global bounds on $|u(s)|_{0}$ and $|\lambda(s)|$, which will be shown to control $|u(s)|_{3+\alpha}$. Without $\eqref{wgrowth2}$ and $\eqref{bcond}$, we would be left with the alternatives (i) $\sup_{s\geq 0}|u(s)|_{0} \to \infty$ or $\sup_{s\geq 0}\lambda(s) \to \infty$; (ii) broadening occurs; or (iii) there is a loss of ellipticity in the limit, which may or may not coincide with either $\lambda(s) \to 0$ or $\lambda(s) \to \infty$. It seems that some restrictions on the growth of $\mathcal{W}$ and $b$ are required for a satisfactory global theory. Perhaps by another set of assumptions on $\mathcal{W}$ and $b$ not utilized here one could force a blow-up in either $|u(s)|_{0}$ or $\lambda(s)$.  

\end{subsection}

\begin{subsection}{History}
Let us briefly recall some of the relevant history. Healey and Simpson obtained global branches of static equilibria for a non-linear elastic mixed boundary value problems in \cite{Healey_Simpson}. This general theory includes alternatives such as a loss of ellipticity, failure of compatibility conditions, or a return to the trivial branch of solutions. Healey and Rosakis \cite{healey1997unbounded} construct unbounded solution branches, which are sometimes referred to as ``solutions in the large." This theory leaves open the possibility that the loading parameter or the norm of the deformation grows arbitrarily large as one follows the global curve. Each of these works are concerned with compressible elastics. Recently, Healey developed global bifurcation results for nonlinear incompressible elastics with conclusions similar to \cite{healey1997unbounded}. The displacements and domains in the theories mentioned above are more general than the ones used in this paper. However, note that each of these works are concerned with bounded domains. Because our problem is posed on an infinite cylinder, there are serious additional complications due to the lack of compactness properties for the underlying PDE. This difficulty is overcome with the analytic global bifurcation theory presented in \cite{chen2018existence}. These authors also recently developed a center manifold reduction to construct small solutions to the anti-plane shear problem in \cite{chen2019center}, which we use for our local bifurcation theory. We also mention the work of \cite{rabier2001global,gebran2009global}, which treat quasilinear elliptic PDE on the whole space using degree theoretic global bifurcation theory. In contrast to \cite{chen2018existence}, these authors impose assumptions that ensure local properness. Because broadening represents a loss of local properness we find the approach of \cite{chen2018existence} to be more natural in this context. 

A word is in order about our choice of domain. Firstly, we are interested in developing a global theory for homoclinic type solutions. Demonstrating broadening behavior is also of great interests to us, which requires the existence of a sequence of spatially localized functions whose effective supports grow without bound. A natural setting for either such analysis is an infinite cylinder. Moreover, there has been a considerable amount of work regarding exponential decay estimates for anti-plane shear on semi-infinite strips of the form $(0,\infty)\times(0,h)$; see \cite[Section 6]{Horgan_antiplane} and the reference therein for a good overview. Thus, interest in anti-plane shear deformations in unbounded domains is well established. 
\end{subsection}

\begin{subsection}{Preliminaries}
Let us fix some notation that will be used in the remainder of the paper. First, for $k \in \mathbb{N}$ and $\alpha \in (0,1)$, let  
\begin{gather*}
    C^{k+\alpha}_{\text{b}}(\overline{\Omega}) := \{ u \in C^{k}(\overline{\Omega}) \; : \; |u|_{k+\alpha}<\infty \},
\end{gather*}
where $C^{k}(\overline{\Omega})$ denotes the space of functions which are $k$ times continuously differentiable on $\Omega$ up to the boundary, and $|\cdot |_{k+\alpha}$ is the usual H\"{o}lder norm. Much of our analysis will be concerned with solutions whose derivatives decay uniformly to $0$, which leads us to consider: 
\begin{gather*}
    C^{k+\alpha}_{0}(\overline{\Omega}) := \left\{ u \in C_{\text{b}}^{k+\alpha}(\overline{\Omega}) \; : \; \lim_{r\to \infty}\sup_{|x|=r} |\partial_{\beta}u(x)|=0, \; \; 0\leq |\beta| \leq k\right\}.
\end{gather*}
Next, we define the Banach spaces $X$ and $Y$ by \begin{gather}
    X:= \{u \in C^{3+\alpha}_{\text{b,e}}(\overline{\Omega})\cap C^{2}_{0}(\overline{\Omega}) : u= 0 \; \text{on} \; \partial \Omega\}
\end{gather}
and 
\begin{gather}
    Y:= C_{\text{b,e}}^{1+\alpha}(\overline{\Omega})\cap C^{0}_{0}(\overline{\Omega})
\end{gather}
where the subscript e denotes evenness in the $x$ and $y$ variables. Equation $\eqref{cub}$ can be written in operator form as 
\begin{gather*}
    \mathcal{F}(u, \lambda)=0,
\end{gather*}
where $\mathcal{F}: X \times \mathbb{R} \to Y$ is real analytic. We will show that $\mathcal{C}^{I,II} \subset X \times \mathbb{R}$ (here, and in the sequel, $\mathcal{C}^{I,II}$ will be used to indicate that a statement holds for both $\mathcal{C}^{I}$ and $\mathcal{C}^{II}$). A detailed investigation of the linearized operator $\mathcal{F}_{u}$ along a local curve of solutions to \eqref{cub} is required in order to establish the existence of $\mathcal{C}^{I,II}$. The following spaces are useful for this task: 
\begin{gather}
    X_{\text{b}}:= \{ u \in C_{\text{b}}^{3+\alpha}(\overline{\Omega}) : u=0 \; \text{on} \; \partial \Omega\} \;\; \text{and} \;\; Y_{\text{b}}:=C^{1+\alpha}_{\text{b}}(\overline{\Omega}).
\end{gather}
Similarly, let $X_{0}:=\{ u \in C^{3+\alpha}_{0}(\overline{\Omega}): u =0 \; \text{on} \; \partial \Omega\}$ and $Y_{0}:= C_{0}^{1+\alpha}(\overline{\Omega})$. Finally, we define an exponentially weighted space that plays a role in the local bifurcation theory. The norm for this space is 
\begin{gather*}
    |f|_{C^{k+\alpha}_{\mu}(\Omega)} := \sum\limits_{\beta \leq k}| w_{\mu} \partial^{\beta}f|_{C^{0}} + \sum\limits_{|\beta|=k}|w_{\mu}|\partial^{\beta}f|_{\alpha}|_{C^{0}}
\end{gather*}
where $k \in \mathbb{n}$, $\alpha \in (0,1)$, $\mu \in \real$, and $w_{\mu}(x) := \text{sech}(\mu x)$. We may then define 
\begin{gather*}
    C^{k+\alpha}_{\mu}(\overline{\Omega}):= \{ f \in C^{k+\alpha}(\overline{\Omega}) \; : \; |f|_{C^{k+\alpha}_{\mu}(\Omega)} < \infty \}.
\end{gather*}
\end{subsection}
\begin{subsection}{Plan of the article}
In Section~\ref{localtheory}, we recall the existence of a local curve of solutions to \eqref{cub} established in \cite[Section~3]{chen2019center} and then prove some monotonicity and symmetry properties that will be important for the later analysis. This section ends by showing the linearized operator along the local curve is invertible; a fact which is essential to the global theory. In Section~\ref{globalsection}, we apply the global bifurcation theory of \cite{chen2018existence} to the local solutions. We also show that the monotonicity properties of the local curve are preserved. Bounds on $|u(s)|_{3+\alpha}$ and $\lambda$ are derived in Section~\ref{UR} through elliptic estimates, monotonicity, and a conserved quantity. Finally, in Section~\ref{main_proofs} we stitch together the previous work and systematically eliminate alternatives of the global bifurcation theorem to prove both \thref{thm1} and \thref{thm2}. 
\end{subsection}

\end{section}

\begin{section}{Local bifurcation}
\label{localtheory}
Our ultimate aim to construct non-peturbative solutions, but first this will require us to refine our understanding of the local theory. After establishing the existence of a local curve, we show some monotonicity and symmetry properties, which will be extended to the global curve in Section~\ref{mono}. The proof of existence for $\mathcal{C}^{I,II}$ will rely upon invertibility properties of the linearized operator $\mathcal{F}_{u}$ along the local curve; these are investigated at the end of the section. 

\subsection{Existence and uniqueness}
In \cite[Section 3]{chen2019center}, the authors establish the existence of a local curve of homoclinic solutions to $\eqref{cub}$ bifurcating from $(u,\lambda)=(0,0)$ under the assumptions 
\begin{empheq}[]{align}
\label{specialcase}
\begin{split}
     b(u,\lambda)=&(\lambda-1)u +b_{1}u^{3}\\  
     \mathcal{W}'(q)=&1+2c_{1}q
\end{split}
\end{empheq}
with $b_{1}+2c_{1}<0$. When this inequality is reversed, front-type solutions are instead obtained. The authors extend this argument to deal with more generalized $b$, including the form \eqref{b}, in \cite[Appendix~B.1]{chen2019center}. Those arguments can also be used to show the existence of local solutions under the more general assumptions of Model I and Model II. This is the content of the next theorem.  
\begin{theorem}
\thlabel{existenceuniqueness}
There exists an $\epsilon_{0}>0$ and a local $C^{0}$ curve
\begin{gather*}
    \mathcal{C}^{I,II}_{\textup{loc}}=\{(u^{\epsilon},\epsilon^{2}) \; : \; 0<\epsilon<\epsilon_{0}\} \subset X_{0}\times\mathbb{R}
\end{gather*}
of solutions to \eqref{cub}, corresponding to Model I or Model II, with the asymptotics
\begin{gather}
\label{smallsol}
    u^{\epsilon}(x,y) = a_{1}\epsilon\textup{sech}(\epsilon x)\cos(y)+O(\epsilon^{2})        \qquad \text{in} \; \; C^{3}_{\textup{b}}(\overline{\Omega}) 
\end{gather}
where $a_{1} = \dfrac{2}{\sqrt{3|b_{2}+2c_{1}|}}$. 
\end{theorem}
\begin{proof}
We reparametrize with $\lambda =\epsilon^{2}$ for convenience. As mentioned above, an existence result was obtained in \cite[Section~3]{chen2019center} under the conditions \eqref{specialcase}. We will follow closely that proof and focus on the places where deviations are necessary to accommodate the more general form of $\mathcal{W}$ we consider. 

Let $L:= \mathcal{F}_{u}(0,0)$ and $L'$ be defined as the restriction of $L$ to $x$-independent functions ($L'$ is called the transversal linearized operator). The center manifold reduction result in \cite{chen2019center} requires that $0$ is a simple eigenvalue of $L'$. The operator $L$ corresponding to \eqref{specialcase}, Model I, or Model II is simply $\Delta+1$ as seen by the structure of \eqref{analytic_assumptions} and \eqref{specialcase}. Clearly $L'$ satisfies the requirements mentioned above. Now, the center manifold reduction given by \cite[~Theorem 1.1]{chen2019center} shows that solutions of \eqref{cub}, that lie in a sufficiently small neighborhood of the origin in $C_{\text{b}}^{2+\alpha}(\overline{\Omega})\times \mathbb{R}$ can be expressed as  
\begin{gather}
\label{smallu}
    u(x+\tau,y) = v(x)\varphi_{0}(y)+v'(x)\tau\varphi_{0}(y)+\Psi(v(x),v'(x), \epsilon)(\tau,y), 
\end{gather}
where $v(x):= u(x,0)$, $\varphi_{0}(y)$ generates the kernel of $L'$, and $\Psi: \real^{3} \to C^{3+\alpha}_{\mu}(\overline{\Omega})$ is a $C^{4}$ coordinate map. Here $\mu>0$ is a positive constant depending on the largest non-zero eigenvalue of $L'$. Moreover, if $(u,\epsilon^{2}) \in C^{3+\alpha}_{\text{b}}(\overline{\Omega})\times \mathbb{R}$ is any sufficiently small solution to \eqref{cub}, then, by \cite[~Theorem 1.1]{chen2018existence}, $v$ solves the second-order ODE 
\begin{gather}
\label{ODEv}
    v''=f(v,v',\epsilon^{2}), \qquad \text{where} \qquad f(A,B,\epsilon^{2}):= \dfrac{d}{dx^{2}}\bigg\rvert_{x=0}\Psi(A,B,\epsilon)(x,0).
\end{gather}
Thus, we are left to show that the change in $\Psi$ resulting from the conditions of Model I or Model II does not affect the existence or general form of $u^{\epsilon}$ in \eqref{smallsol}. Let us point out that $\Psi$ inherits the following symmetry properties from the original PDE \eqref{cub}: 
\begin{gather}\label{psi_sym}
    \Psi(-A,-B,\epsilon) = -\Psi(A,B,\epsilon) \qquad \text{and} \qquad \Psi(A,-B,\epsilon)(-x,y) = \Psi(A,B,\epsilon)(x, y).
\end{gather}
From \eqref{ODEv} it follows that 
\begin{gather} \label{f_sym}
    f(-A,-B) = - f(A,B) \qquad \text{and} \qquad f(A,-B) = f(A,B). 
\end{gather}

To derive an expression for $f$, we exploit \cite[Theorem 1.2]{chen2019center} to conclude that $\Psi$ admits a Taylor expansion of the form   
\begin{gather}
\label{psiexpand}
    \Psi(A,B,\epsilon) = \sum_{\mathcal{J}}\Psi_{ijk}A^{i}B^{j}\epsilon^{k}+\mathcal{R},
\end{gather}
where the index set $$\mathcal{J} =\{(i,j,k) \in \mathbb{N} \; : \; i+2j+k \leq 3, i+j+k \geq 2, i+j \geq 1\},$$ the coefficients $\Psi_{ijk}\in C_{\mu}^{3+\alpha}(\overline{\Omega})$, and the error term $\mathcal{R}$ is of order 
$O((|A|+|B|^{1/2}+\epsilon)^4)$ in $C^{3+\alpha}_{\mu}(\overline{\Omega})$. 

Combining \eqref{smallu} and \eqref{psiexpand} yields 
\begin{gather}\label{u_expand}
    u(x,y) = (A+Bx)\varphi_{0}(y)+\sum_{\mathcal{J}}\Psi_{ijk}A^{i}B^{j}\epsilon^{k}+\mathcal{R},
\end{gather}
where $A=v(0)$ and $B=v'(0)$. For a fixed $i,j,k$ the general theory now allows us to solve for $\Psi_{ijk}$ via a hierarchy of equations of the form
\begin{empheq}[left=\empheqlbrace]{align}
\label{psi_system}
\begin{split}
       &L(\Psi_{ijk})=F_{ijk} \\
       & Q(\Psi_{ijk})=0,
\end{split}
\end{empheq}
where $Q$ is the projection onto $\text{ker}L'$. The $F_{ijk}$ terms are obtained by iteratively feeding truncations of \eqref{u_expand} into $\mathcal{F}^{r} - L$ where $\mathcal{F}^{r}$ is $\mathcal{F}$ precomposed with a certain cutoff function. The key point here is that the $Q$ is unchanged by our modification of $\mathcal{W}$, and the $F_{ijk}$ terms are independent of terms of the order $O(|A|+|B|^{1/2}+\epsilon)^4)$ in $C^{3+\alpha}_{\mu}(\overline{\Omega})$. Our generalized $\mathcal{W}$ introduces, for example, the extra nonlinear term $c_{2}\nabla\cdot(|\nabla u|^{4}\nabla u)$ into \eqref{cub} near $(u,\lambda)=(0,0)$. We see that applying this to \eqref{u_expand} yields only terms of order $O((|A|+|B|^{1/2}+\epsilon)^4)$. Hence, from this point on the argument for existence of solutions to \eqref{cub} carries through without change. In particular, one can solve for $\Psi_{ijk}$ in the exact manner presented in \cite[Section~3.1]{chen2019center} and \cite[Appendix~B.1]{chen2019center}.

Although the rest of the argument now follows verbatim from \cite{chen2019center}, we continue the sketch because it will help to explain some later reasoning. Having calculated $\Psi_{ijk}$, we find that $f$ takes the form
\begin{gather}
\label{f}
    f(A,B,\epsilon) = \epsilon^{2}A+\dfrac{3(b_{1}+2c_{1})}{4}A^{3}+r(A,B,\epsilon)
\end{gather}
where $r \in C^3$ is an error term of the order $O(|A|(|A|+|B|^{1/2}+\epsilon)^{3}+|B|(|A|+|B|^{1/2}+\epsilon)^{2})$. Using the re-scaled variables 
\begin{gather*}
    x = : X/\epsilon , \qquad v(x) = : \epsilon V(x), \qquad v_{x}(x) =: \epsilon^{2}W(x)
\end{gather*}
we may now write \eqref{ODEv} as the planar system
\begin{gather}
    \label{system}
        \begin{cases}
                    V_{X}=W\\ 
                    W_{X} = V-a_{1}^{-2}V^{3}+R(V,V,\epsilon)
                \end{cases}.
\end{gather}
where the rescaled error $R(V,W,\epsilon) = O(|\epsilon|(|V|+|W|)$.  When $\epsilon=0$ the system has the explicit homoclinic orbit 
\begin{gather}
    V=a_{1}\text{sech}(X) \qquad W=-a_{1}\text{sech}(X)\text{tanh}(X) . 
\end{gather}
This solutions crosses the $V$-axis transversely. Since \eqref{system} has the reversal symmetries 
\begin{gather*}
(V(X),W(X)) \mapsto (V(-X),-W(-X))\qquad \text{and} \qquad (V(X), W(X)) \mapsto (V(X),-W(X)), 
\end{gather*}
which it inherits from \eqref{f_sym} and \eqref{f}, this intersection will persist for small $\epsilon$, so we obtain a family of homoclinic solutions. Undoing the scaling and appealing to \cite[~Theorem 1.1]{chen2019center} shows that the family \eqref{smallsol} are indeed solutions to \eqref{cub}. 
\end{proof}

We now establish some qualitative properties of small solutions to \eqref{cub}. 

\begin{theorem}
\thlabel{Positive}
Suppose that $(u, \epsilon^{2}) \in X_{0} \times \real$ is a solution to \eqref{cub} under the assumptions of Model I or Model II. There exists $\delta_{0}$ such that if $|u|_{3+\alpha}+\epsilon^{2} < \delta_{0}$, then $(u, \epsilon^{2}) \in \mathcal{C}^{I,II}_{\textup{loc}}$ after a possible translation or reflection in $x$. Moreover, if $(u,\epsilon^{2})\in \mathcal{C}^{I,II}_{\textup{loc}}$, then $u$ is even in $x$ and $y$ and monotone in that $u_{x} < 0$ for $x>0$. 
\end{theorem}
\begin{proof}
First, we show there exists $\delta_{0}$ small enough to ensure $u>0$. The Malgrange preparation theorem allows us to write $b(\lambda,z)=zw(\lambda ,z)$ for a smooth $w$ defined in some neighborhood of $(0,0)$, see for example \cite[Theorem~7.1]{chow}. Then, \eqref{cub} becomes 
\begin{empheq}[left=\empheqlbrace]{align}
\label{cubmod}
\begin{split}
        \mathfrak{a}_{1}u_{xx}+\mathfrak{a}_{2}u_{yy}+4\mathcal{W}''(|\nabla u|^2)u_{x}u_{xy}u_{y}-uw(\lambda, u)&=0 \qquad \text{in} \; \Omega\\  
         u&=0 \qquad \text{on} \; \partial \Omega. 
\end{split}
\end{empheq}
where 
\begin{gather*}
\mathfrak{a}_{1}=\mathcal{W}'(|\nabla u|)+2\mathcal{W}''(|\nabla u|^2)(u_{x})^{2} \qquad \text{and} \qquad \mathfrak{a}_{2}=\mathcal{W}'(|\nabla u|)+2\mathcal{W}''(|\nabla u|^2)(u_{y})^{2}.
\end{gather*}
Thus $\mathfrak{a}_{1}$ and $\mathfrak{a}_{2}$ are uniformly positive for small enough $\delta_{0}$ (note that this is only a concern for Model II, since \eqref{Wcond} ensures a such a lower bound holds for Model I). We write \eqref{cubmod} this way in order to view it as a linear elliptic PDE and apply a comparison principle argument. 

Consider the function
\begin{gather}
    \Phi^{\delta}=\Phi^{\delta}(y): = \log{(2+\sqrt{\delta}y)}\cos(\sqrt{1-\lambda}y).
\end{gather}
An elementary calculation reveals that 
\begin{empheq}{align}
\label{supsol1}
\begin{split}
    a_{2}\Phi^{\delta}_{yy}-w(\lambda, u)\Phi^{\delta} = \Big(&\dfrac{-\epsilon\cos(\sqrt{1-\lambda}y)}{(2+\sqrt{\epsilon}y)^2}-\dfrac{2\sqrt{\epsilon(1-\lambda)}\sin(\sqrt{1-\lambda}y)}{2+\sqrt{\epsilon}y}\Big)(1+O(\epsilon^{2})) \\ 
    &+O(\epsilon^{2}\cos(\sqrt{1-\lambda}y)) \;\; \text{in} \; \; C^{0}(\overline{\Omega}),
\end{split}
\end{empheq}
where we have used the asymptotics of $b$ and $\mathcal{W}$ in \eqref{analytic_assumptions}, which hold for small enough $\epsilon_{0}$. The right hand side of \eqref{supsol1} is strictly negative whenever $0\leq y \leq \frac{\pi}{2}$ and  $\delta$ is small enough. Moreover, $\Phi^{\delta}>0$ in $\overline{\Omega}$. So, if we establish non-negative boundary values for $u$ on the region $(-\infty,\infty)\times(0,\frac{\pi}{2})$, then we may invoke the maximum principle for uniformly elliptic operators with a positive super-solution (see \hyperref[pos_sup_sol]{\thref{max_pricp}.\ref{pos_sup_sol}}) to conclude that $u>0$ on $\mathbb{R}\times (0,\frac{\pi}{2})$. 

We already know $u=0$ on $\real\times\{\frac{\pi}{2}\}$, and  a phase plane analysis will show $u>0$ on $\real \times \{0\}$. Indeed, $v:=u(x,0)$ solves the ODE \eqref{ODEv} by \cite[~Theorem 1.1]{chen2019center}. If we write this as a planar system, which has the same structure as \eqref{system}, then the symmetries $(V,W) \mapsto (-V,-W)$ and $(V,W) \mapsto (V(-X),-W(-X))$ imply that a homoclinic orbit that intersects the positive $V$ axis meets the $W$ axis only at $(0,0)$. Hence, after a possible reflection $u(x,0)>0$, so $u>0$ for $0<y <\frac{\pi}{2}$ by the remarks at the end of the previous paragraph. Redoing the above analysis with $\Phi^{\delta}(-y)$ shows $u>0$ in $\Omega$. 

Now that the positivity of $u$ is established, we find from a moving planes argument in \cite[Theorem~3.2]{MR1113099} that $u$ is even in $x$ about some line $x=x_{1}$ with $u_{x}<0$ for $x>x_{1}$. The translation $x \mapsto x-x_{1}$ sends $u$ to a positive solution of \eqref{cub} with the desired monotonicity and evenness properties in $x$. The phase plane analysis for \eqref{system} in  \thref{existenceuniqueness} shows that $u^{\epsilon}_{x}(0,y)=0$, where $u^{\epsilon} \in \mathcal{C}^{I,II}_{\text{loc}}$, whence it follows that $u^{\epsilon}$ is even about $x=0$. 

Observe that the previous paragraphs established the uniqueness of small solutions to \eqref{cub} up to translations and reflections in $x$. In particular, the elements of $\mathcal{C}^{I,II}_{\text{loc}}$ are the unique positive and even solutions to \eqref{cub} in a sufficiently small neighborhood of $(0,0)$ in $X_{0}\times\real$. Finally, the elements of $\mathcal{C}^{I,II}_{\text{loc}}$ must be even in $y$ since the reflection $y \mapsto -y$ will take an element of $\mathcal{C}^{I,II}_{\text{loc}}$ to another positive solution that is even and monotone in $x$.  
\end{proof}

\subsection{Linearized problem}
In this section, we show the linearized operator $\mathcal{F}_{u}(0,\lambda): X \to Y$ is invertible for $0<\lambda\leq1$. This fact plays an important role in the analysis to follow. In particular, it implies the Fredholmness of $\mathcal{F}: X \to Y$, which will extend to the global curve. A simple calculation yields 
\begin{gather}
\label{LO}
    \mathcal{F}_{u}(0,0) = \Delta+1
\end{gather}
for Model I or Model II. The notion of a limiting operator is needed for the next two lemmas. If 
\begin{gather*}
L = a_{ij}(x,y)\partial_{x_{i}}\partial_{x_{j}}+b_{i}(x,y)\partial_{x_{i}}+c(x,y),
\end{gather*}
and as $x \to \pm \infty$ we have 
\begin{gather*}
    a_{ij}(x,y) \to \tilde{a}_{ij}(y), \;\; b_{i}(x,y) \to \tilde{b}_{i}(y), \;\; c(x,y) \to \tilde{c}(y), 
\end{gather*}
where each of $\tilde{a_{ij}}, \tilde{b}_{i}$, and $\tilde{c}$ belongs to $C^{\alpha}_{\text{b}}[-\frac{\pi}{2},\frac{\pi}{2}]$, then the limiting operator $\tilde{L}$ is defined as 
\begin{gather*}
    \tilde{L}: = \tilde{a}_{ij}\partial_{x_{i}}\partial_{x_{j}}+\tilde{b}_{x_{i}}\partial_{x_{i}}+\tilde{c}.
\end{gather*}
\begin{lemma}[Invertibility of linearized operator at 0]
\thlabel{linearized at 0}
For all $0<\lambda \leq 1$, $\mathcal{F}_{u}(0,\lambda): X \to Y$ is invertible. 
\end{lemma}
\begin{proof}
Fix $0<\lambda\leq1$ and let $\varphi(y)=\cos(\sqrt{1- \lambda}y)$. Since $\varphi>0$ on $[-\frac{\pi}{2},\frac{\pi}{2}]$, we may write $u=:v\varphi$, so that $\mathcal{F}_{u}(0,\lambda)u=0$ implies 
\begin{empheq}[left=\empheqlbrace]{align} \label{divisiontrick}
\begin{split}
      \Delta v + \dfrac{2\varphi_{y}}{\varphi}v_{y} &=0 \qquad \text{in} \; \Omega\\  
        v&=0 \qquad \text{on} \; \partial \Omega. 
\end{split}
\end{empheq}
Let $L$ be the linear operator associated with~\eqref{divisiontrick} which acts on $v$. Note that $L:X_{\text{b}} \to Y_{\text{b}}$ has trivial kernel by the strong maximum principle (\hyperref[strong_max]{\thref{max_pricp}.\ref{strong_max}}). For $\gamma \in \mathbb{R}$, let $L_{\gamma}=L-\gamma$ and denote by $\mathcal{B}$ the corresponding bilinear operator:
\begin{gather*}
    \mathcal{B}[w,w] = \int_{\Omega} \left(|\nabla w|^{2} - \dfrac{\varphi_{y}}{\varphi}ww_{y}+\gamma w^{2}\right)\,dx\,dy= 
    \int_{\Omega}\left( |\nabla w|^{2} +\left(\dfrac{\partial}{\partial y}\dfrac{2\varphi_{y}}{\varphi}\right)w^{2}+\gamma w^{2}\right)\,dx\,dy,
\end{gather*}
for $w \in H_{0}^{1}$. When $\gamma$ is large enough, $\mathcal{B}$ is coercive and hence Lax--Milgram implies $L_{\gamma}:H_{0}^{1} \to L^{2}$ is invertible. 

We will next show that $L_{\gamma}: X_{\text{b}} \to Y_{\text{b}}$ is invertible. The argument is similar to the one found in \cite[Appendix A.2]{wheeler2013large}. Let $\rho_{\epsilon}(x):=\text{sech}(\epsilon x)$. Conjugating by $\rho_{\epsilon}$ the problem $L_{\gamma}=f$ may be transformed into the equivalent one 
\begin{gather*}
    L_{\gamma}^{\epsilon}u_{\epsilon}=L_{\gamma}u_{\epsilon}-\dfrac{2\partial_{x}\rho_{\epsilon}}{\rho_{\epsilon}}\partial_{x}u_{\epsilon}+\left(\dfrac{\partial_{x}^{2}\rho_{\epsilon}}{\rho_{\epsilon}}-\dfrac{2(\partial_{x}\rho_{\epsilon}^{2})}{\rho_{\epsilon}^{2}}\right)u_{\epsilon}=f_{\epsilon}
\end{gather*}
where $u_{\epsilon}:=u\rho_{\epsilon}$ and $f_{\epsilon}:=f\rho_{\epsilon}$. If $f \in Y_{\text{b}}$ then $f_{\epsilon} \in L^{2}$, and the equation $L_{\gamma}(u_{\epsilon})=f_{\epsilon}$ is solvable by the work above. Note that 
\begin{gather*}
    \|L_{\gamma}^{\epsilon}-L_{\gamma}\|_{X_{\text{b}} \to Y_{\text{b}}} = \left\|\dfrac{2\partial_{x}\rho_{\epsilon}}{\rho_{\epsilon}}\partial_{x}+\left(\dfrac{\partial_{x}^{2}\rho_{\epsilon}}{\rho_{\epsilon}}-\dfrac{2(\partial_{x}\rho_{\epsilon}^{2})}{\rho_{\epsilon}^{2}}\right)\right\|_{X_{\text{b}}\to Y_{\text{b}}} \longrightarrow 0, \;\;\; \text{as} \;\;\;\epsilon\to 0,
\end{gather*}
so for small enough $\epsilon_{0}$, the pertubation $L_{\gamma}^{\epsilon}$ of $L_{\gamma}$ remains invertible whenever $0<\epsilon<\epsilon_{0}$.

From \cite[Theorem 8.8]{gilbarg2015elliptic} and \cite[Theorem 9.19]{gilbarg2015elliptic}, we know $u_{\epsilon} \in C^{3+\alpha}(\overline{\Omega})\cap C_{\text{b}}^{\alpha}(\overline{\Omega})$. Moreover, By Schauder estimates and injectivity, we have the bound 
\begin{gather*}
|u_{\epsilon}|_{2+\alpha}\leq C|f_{\epsilon}|_{\alpha},   
\end{gather*}
wehre $C>0$ is independent of $\epsilon$. Therefore, we are able to extract a subsequence $\epsilon_{n}\to 0$ for which $u_{\epsilon_{n}}\to u$ in $C_{\text{loc}}^{2}(\overline{\Omega})$ with $u \in C_{\text{b}}^{2+\alpha}(\overline{\Omega})$. Letting $n \to \infty$ in the above equation we find the $L_{\gamma}u=f$.  

Now that the invertibility of $L_{\gamma}:X_{\text{b}}\to Y_{\text{b}}$ has been established, we will make use of the continuity of the Fredholm index to conclude that $L: X_{\text{b}} \to Y_{\text{b}}$ is invertible. Let $L_{\gamma t}:= L -t\gamma$. It is clear that $L_{\gamma t}$ is its own limiting operator for $t \in [0,1]$, since its coefficients are $x$-independent. The limiting problem has no non-trivial solutions because $L_{t\gamma}$ satisfies the strong maximum principle (\hyperref[strong_max]{\thref{max_pricp}.\ref{strong_max}}). Lemma A.8 of \cite{wheeler2015solitary} now shows $L_{t\gamma}$ must be semi-Fredholm with index $<\infty$. Thus, the Fredholm index must then be preserved along the family $\{L_{t\gamma}\}_{t \in [0,1]}$. We can now conclude that $L:X_{\text{b}} \to Y_{\text{b}}$ has Fredholm index $0$, just as the operator $L_{\gamma }$. Hence, $L: X_{\text{b}} \to Y_{\text{b}}$ is in fact invertible since it also has a trivial kernel. From \cite[Lemma A.12]{wheeler2015solitary}, $L:X_{0}\to Y_{0}$ must also have Fredholm index $0$, and again the kernel is trivial so that $L:X_{0} \to Y_{0}$ is invertible. Finally, it is not hard to see from the structure of $L$ that data $f \in Y \subset Y_{0}$ must have a corresponding solution $u \in X$. For example, if $f$ is even in $y$, and $v$ is the unique solution to $Lv=f$, then a quick check shows that $Lv(x,-y) =f$ as well. Hence, $L: X \to Y$ is invertible.  
\end{proof}

Now consider the linearized operator $\mathcal{F}_{u}(u,\lambda)$ with $(u,\lambda) \in \mathcal{C}^{I,II}_{\text{loc}}$. We know $\mathcal{F}_{u}(u,\lambda)\partial_{x} u=0$ by translation invariance and elliptic regularity. Thus, $\mathcal{F}_{u}(u,\lambda)$ has nontrivial kernel acting on $X_{\text{b}}$. However, if we instead restrict to $X$, which by definition imposes even symmetry, then we will have injectivity.      

\begin{lemma}[Trivial kernel]
\thlabel{Trivial Kernel}
 For all $(u, \lambda) \in \mathcal{C}_{\textup{loc}}^{I,II}$, $\mathcal{F}_{u}(u,\lambda): X \to Y$ is injective, whenever $(u,\lambda) \in \mathcal{C}^{I,II}_{\textup{loc}}$.
\end{lemma}
\begin{proof}
From \cite[Theorem~1.6]{chen2019center} and \cite[~Appendix B.1.]{chen2019center}, if $\dot{u}\in C_{\text{b}}^{3+\alpha}(\overline{\Omega})$ is a solution of $\mathcal{F}_{u}(u,\lambda)\dot{u}=0$, then $\dot{v}:=\dot{u}(\cdot,0)$ solves the linearized reduced ODE 
\begin{gather}
\label{reduction}
\dot{v}''=r_{B}\dot{v}'+(\lambda+\frac{9(b_{1}+2c_{2})}{4}v^{2}+r_{A})\dot{v} 
\end{gather}
where $v:= u(\cdot,0)$. As noted above, $\partial_{x}u$ is in the kernel of $\mathcal{F}_u(u,\lambda)$, so $v_{x}$ is an odd and bounded solution to \eqref{reduction}. Suppose that we had another bounded solution $w \in C^{2}_{\text{b}}(\mathbb{R})$ to $\eqref{reduction}$ that is linearly independent of $v$. From Abel's identity
\begin{gather*}
    W(x)=W(0)\exp{\left(\int_{0}^{x} \text{tr}(P(s))\, ds\right)}
\end{gather*}
where $W(x)$ is the Wronskian of $v$ and $w$ evaluated at $x$, and $P$ is the matrix defined by
\begin{gather*}
P\;:=
\begin{pmatrix}
0 & 1 \\ 
\lambda+\frac{9(b_{1}+2c_{1})}{4}c_{1}v^{2}+r_{A}(v,v',\epsilon) & r_{B}
\end{pmatrix}.
\end{gather*}
Since $u_{x}, u_{xx}, w,$ and $w_{x}$ are all bounded, and $u_{x},u_{xx}$ each decay at infinity, we see that
\begin{gather*}
|\text{det}W(x)|\leq (w^{2}(x)+w_{x}^{2}(x))\cdot (u_{x}^{2}(x,0)+u_{xx}(x,0)) \to 0 \;\;\; \text{as} \;\;\;|x| \to \infty. 
\end{gather*}
But then must have 
\begin{gather}\label{abel_contr}
    \int_{0}^{x}r_{B}(u_{x}(t,0),u_{xx}(t,0))\,dt \to -\infty \;\;\; \text{and} \;\;\; \int_{-x}^{0}r_{B}(u_{x}(t,0),u_{xx}(t,0))\,dt \to \infty \;\;\; \text{as} \;\;\; x \to \infty.
\end{gather}
Recalling the symmetry properties of $f$ in \eqref{f_sym} and the explicit form given in \eqref{f}, it follows that 
\begin{gather*}
    r_{B}(A,B)=-r_{B}(A,-B)=r_{B}(-A,-B).
\end{gather*}
This would imply 
\begin{align}\label{abel_contr2}
\begin{split}
   \lim_{x \to \infty}\int_{-x}^{0}r_{B}(u_{x}(t,0),&u_{xx}(t,0))\,dt = \lim_{x\to \infty}-\int_{0}^{x}r_{B}(u_{x}(-t,0),u_{xx}(-t,0))\,dt \\ 
        &=\lim_{x\to \infty}\int_{0}^{x}r_{B}(u_{x}(t,0),u_{xx}(t,0))\,dt, 
\end{split}
\end{align}
where we used the properties of $r_{B}$, oddness of $u_{x}$ and evenness of $u_{xx}$. Equations \eqref{abel_contr} and \eqref{abel_contr2} together force a contradiction.  
Hence, there cannot be two linearly independent bounded solutions to $\eqref{reduction}$. 

At this point we may conclude that $v_{x}$ generates the solution set of \eqref{reduction}. Thus, $\mathcal{F}_{u}:X \to Y$ has trivial kernel, since any non-zero element would necessarily be odd. To see this, suppose, by a slight abuse of notation, that some $w(x,y) \in C^{3}_{\text{b}}(\overline{\Omega})$ satisfies $\mathcal{F}_{u}(u,\lambda)w =0$. Recall that \cite[Theorem~1.1]{chen2019center} gives the expansion
\begin{gather*}
    w(x,y) = \varphi_{0}(x)w(x,0)+\Psi(w(x,0),w_{x}(x,0),\lambda)(0,y),
\end{gather*}
where $w(x,0)$ is odd in $x$ by the work above. The symmetries in \eqref{psi_sym} imply the additional symmetry 
\begin{gather*}
    \Psi(-A,B,\lambda)(0,y) = -\Psi(A,B,\lambda)(0,y),
\end{gather*}
from which we may conclude that $w$ is odd in $x$. 
\end{proof}

Finally, we show that $\mathcal{F}_{u}$ is invertible along the local curve. 

\begin{lemma}[Invertibility]
\thlabel{Invertibility}
For any $(u,\lambda) \in \mathcal{C}_{\text{loc}}^{I,II}$, the linearized operator $\mathcal{F}_{u}(u,\lambda): X \to Y$ is invertible.  
\end{lemma}
\begin{proof}
We found that $\mathcal{F}_{u}(u,\lambda): X \to Y$ has trivial kernel whenever $(u,\lambda) \in \mathcal{C}^{I,II}_{\text{loc}}$ in \thref{Trivial Kernel}.  It therefore suffices to show that this operator is Fredholm index $0$. The limiting operator of $\mathcal{F}_{u}(u, \lambda)$ is simply $\mathcal{F}_{u}(0,\lambda)$ because $u$ decays as $x \to \pm \infty$. Recall that $\mathcal{F}_{u}(0,\lambda)$ was shown to be invertible in \thref{linearized at 0}. By \cite[Lemma A.13]{wheeler2015solitary} it follows that the Fredholm indices of $\mathcal{F}_{u}(u,\lambda)$ and $\mathcal{F}_{u}(0,\lambda)$ match. Hence, $\mathcal{F}_{u}(u,\lambda)$ is in fact Fredholm index $0$, and the result follows. 
\end{proof}

\end{section}

\begin{section}{Global bifurcation}
\label{globalsection}
\subsection{Background theory}
We begin this section by recalling some of the global bifurcation theory developed in \cite[Section 6]{chen2018existence}. The results stated here are tailored to the problem at hand. Let $\mathcal{I}=(0,1)$ and 
\begin{align}
\label{ODef}
\begin{split}
    &\mathcal{O} = \bigcup\limits_{\delta>0}\mathcal{O}_{\delta} \qquad \text{where} \\ \mathcal{O}_{\delta} = X \cap \big\{u \in C^{3}(\overline{\Omega})\,:\, &\liminf_{(x,y) \in \overline{\Omega}}\big(\mathcal{W}'(q)+2q\mathcal{W}''(q)\big)\big\vert_{q=|\nabla u(x,y)|^{2}} >\delta \big\}.
\end{split}
\end{align}
\begin{theorem}
\thlabel{global}
There is a curve of solutions $\mathcal{C}^{I,II}\subset \mathcal{F}^{-1}(0)$, where $\mathcal{F}$ corresponds to either Model I or Model II, parameterized as $\mathcal{C}^{I,II} \coloneqq \{ (u(s),\lambda(s)) : 0<s<\infty\} \subset \mathcal{O} \times \mathcal{I}$ with the following properties. 
\begin{enumerate}[label=\rm(\alph*)]
    \item  One of the following alternatives holds. 
    \begin{enumerate}[label=\rm(\roman*)]
        \item \label{blowup_alt}\textup{(Blowup)} As $s \to \infty$
        \begin{gather}
        \label{blowup}
            N(s)\coloneqq |u(s)|_{3+\alpha}+\dfrac{1}{\text{dist}(u(s),\partial \mathcal{O})}+\lambda(s)+\dfrac{1}{\text{dist}(\lambda(s),\partial \mathcal{I})}\to \infty
        \end{gather}
        \item \label{lossofC}\textup{(Loss of compactness)} There exists a sequence $s_{n}\to \infty$ such that $\sup_{n}N(s_{n}) < \infty$ but $\{u(s_{n})\}$ has no subsequences converging in $X$. 
    \end{enumerate}
    \label{alternatives}
    \item Near each point $(x(s_{0}),\lambda(s_{0}))\in \mathcal{C}$, we can reparametrize $\mathcal{C}$ so that $s\mapsto (x(s),\lambda(s))$ is real analytic. 
    \item $(x(s),\lambda(s)) \notin \mathcal{C}_{\textup{loc}}$ for $s$ sufficiently large. \label{notinloc}
\end{enumerate}
\end{theorem}
\begin{proof}
We have shown that the linearized operator is invertible along the local curve and the result follows directly from \cite[Theorem~6.1]{chen2018existence}. 
\end{proof}
Alternative (i) encapsulates several interesting possibilities. We note that a blow-up in \eqref{blowup} can be achieved by a loss of ellipticity, $\lambda$ returning to $0$, or the more obvious unboundedness of $\lambda$ or $|u(s)|_{3+\alpha}$. Throughout the rest of the paper we investigate alternatives (i) and (ii) for Models I and II. This will ultimately lead us to discover that broadening occurs invariably in Model I and that a loss of ellipticity is ensured for Model II. At times we focus on segments of the curve $\mathcal{C}^{I,II}$ of the form
\begin{gather}
    \mathcal{C}^{I,II}_{\delta} := \mathcal{C}^{I,II} \cap \mathcal{O}_{\delta}.
\end{gather}
Note that $\mathcal{C}^{II}=\mathcal{C}^{II}_{\xi_{1}}$ by \eqref{Wcond}. 

At this point, it is convenient to recall another result from \cite{chen2018existence} which helps characterize alternative (ii) of~\thref{global}. 

\begin{theorem}[Chen, Walsh, Wheeler \cite{chen2018existence}]
\thlabel{CorF}
If $\{(u_{n},\lambda_{n})\}$ is a sequence of solutions to \eqref{cub} that is uniformly bounded in $C_{\textup{b}}^{3+\alpha}(\overline{\Omega})\times \real$, with the additional monotonicity property 
\begin{gather}\label{u_even}
    u_{n}(x,y) \; \; \text{is even in $x$ and} \; \; u_{x} \leq 0 \;\; \text{for} \; \; x \geq 0 
\end{gather}
for each $n$ as well as the asymptotic condition 
\begin{gather}\label{U_to_zero}
    \lim_{|x| \to \infty}u_{n}(x,y) = U(y) \; \; \text{uniformly in $y$}
\end{gather}
for some fixed function $U \in C_{\textup{b}}^{3+\alpha}([-\frac{\pi}{2}, \frac{\pi}{2}])$, then either 
\begin{enumerate}[label=(\roman*), font=\upshape]
    \item we can extract a subsequence $\{u_{n}\}$ so that $u_{n} \to u$ in $C_{\textup{b}}^{3+\alpha}(\overline{\Omega})$; or \label{Compactness}
    \item we can extract a subsequence and find $x_{n} \to \infty$ so that the translated sequence $\{\tilde{u}_{n} \}$ defined by $\tilde{u}_{n} = u_{n}(\cdot + x_{n}, \cdot)$ converges in $C_{\textup{loc}}^{3}(\overline{\Omega})$ to some $\tilde{u} \in C_{\textup{b}}^{3+\alpha}(\overline{\Omega})$ that solves \eqref{cub} and has $\tilde{u} \not\equiv U$ with $\tilde{u}_{x} \leq 0$. \label{Front}  
\end{enumerate}
\end{theorem}
Note that this theorem requires some symmetry and monotonicity properties in $u_{n}$. The following subsection demonstrates these properties, and more, for elements of $\mathcal{C}^{I,II}$.

\subsection{Monotonicity and nodal properties} 
\label{mono}
We show that elements of $\mathcal{C}^{I,II}_{\text{loc}}$ exhibit certain qualitative features by using the asymptotics \eqref{smallsol} and maximum principle arguments. In fact, we have already established that \eqref{u_even} and \eqref{U_to_zero} (with $U(y) =0$) hold along $\mathcal{C}^{I,II}_{\text{loc}}$ in \thref{Positive}. Our goal is to prove that these persist along $\mathcal{C}^{I,II}$. The following sets will be useful for our analysis: 
\begin{align}
\begin{split}
    \Omega^{+}&: = \{ (x,y) \in \Omega \; : \; x>0 \}  \\ 
    \Omega_{+}&: = \{(x,y)\; : \; |x|<R, 0<y \leq \frac{\pi}{2} \} \\
    L &:= \{ (0,y) \; : \; -\pi/2 < y < \pi/2 \} \\ 
    T &:= \{ (x, \pi/2) \; : \; 0< x < \infty \} \\ 
    B &:= \{ (x, -\pi/2) \; : \; 0< x< \infty \} \\
    M &:= \{ (x, 0) \; : \; 0\leq x< \infty \}.
\end{split}
\end{align}
The nodal properties we are concerned with are as follows: 
\begin{align}
\label{nodalprop}
\begin{split}
    u_{x}< 0& \; \; \; \text{on} \; \;\Omega^{+}  \\ 
    u_{y}< 0&  \; \; \; \text{on} \; \;\Omega_{+}\\ 
    u_{xx}< 0& \; \; \; \text{on} \; \; L \\ 
    u_{xy}> 0&  \; \; \; \text{on} \; \; T \\ 
    u_{xxy}> 0 \; \; \; \text{at} \; \; (0,\frac{\pi}{2}) \; \; &\text{and} \;\; u_{xxy}> 0 \; \; \; \text{at} \; \; (0,-\frac{\pi}{2}) \\ 
    u_{yy}< 0& \; \; \; \text{on} \; \; M
\end{split}
\end{align}
The reason for such a long list is owed to the style of argument. Roughly speaking, we will split the right half (or upper half) of $\overline{\Omega}$ into a finite rectangle and infinite tail region (or into a finite rectangle and \textit{two} tail regions). The conditions in \eqref{nodalprop} will help gain control on the sign of either $u_{x}$ or $u_{y}$ near the boundary. 
The following result gives a condition which ensures a sign on the $x$ derivative of small solutions to \eqref{cub}. 

\begin{lemma}[Asymptotic monotonicity]\thlabel{asymptotic}
There exists $\epsilon_{0}>0$ such that, if $u \in C^{3}_{\textup{b}}(\overline{\Omega})$ and  $(u,\lambda) \in \mathcal{O}_{\delta} \cap \mathcal{F}^{-1}(0),\;\;\text{for some} \; \;\delta>0$, $\lambda>0$,  $u_{x} \leq 0$ on $L_{x_{0}}: = \{(x,y) \in \Omega\; : \; x = x_{0}\}$, and 
\begin{gather*}
    |u|_{2}< \epsilon_{0},
\end{gather*}
then $u_{x}\leq0$ in $\Omega\cap \{ (x,y)\; : \; x\geq x_{0}\}$. 
\end{lemma}

\begin{proof}
Fix $\lambda$ with $0<\lambda \leq 1$. Differentiating $(\ref{cub})$ with respect to $x$ gives
\begin{empheq}[left=\empheqlbrace]{align}
\label{quasilinearize}
\begin{split}
\nabla \cdot(\mathcal{W}'(|\nabla u|^{2})\nabla v+ 2\mathcal{W}''(|\nabla u|^{2})(\nabla u \otimes \nabla u)\nabla v)-b_{u}(u,\lambda)v&=0\qquad \text{in} \; \Omega\\  
        v&=0 \qquad \text{on} \; \partial \Omega
\end{split}
\end{empheq}
where $v=u_{x}$. We see that $\eqref{quasilinearize}$ is uniformly elliptic by \eqref{Wcond} in the case of Model I, and the fact that $u \in \mathcal{O}_{\delta}$ in the case of Model II. Let $v:=\varphi z$, where 
\begin{gather} \label{mod_u_x}
    \varphi(y) = \cos(\sqrt{k}y)
\end{gather}
and $1-\lambda < k < 1$. After plugging \eqref{mod_u_x} into \eqref{quasilinearize}, we find that $z$ satisfies a uniformly elliptic equation with zeroth order term
\begin{gather}
\label{zeroth}
    \frac{1}{\varphi}(\partial_{y}(\mathcal{W}'(|\nabla u|^{2})\varphi+2\mathcal{W}''(|\nabla u|^{2})(u_{y}^{2}+u_{x}u_{y})\varphi).
\end{gather}
If $\epsilon_{0}$ is chosen small enough, then from \eqref{w} and \eqref{b} it follows that \eqref{zeroth} admits the $C^{0}(\overline{\Omega})$ expansion  
\begin{gather}\label{neg_zeroth}
    \frac{1}{\varphi}((1-\lambda-k)\varphi+O(\epsilon_{0}^{2}))<0.
\end{gather}
Thus, \eqref{neg_zeroth} implies that $z$ satisfies the strong maximum principle (\hyperref[strong_max]{\thref{max_pricp}.\ref{strong_max}}). Note $z = 0$ on $\partial \Omega$, and $z \leq 0$ on $L_{x_{0}}$, so it follows from the maximum principle that $z \leq 0$ in $\overline{\Omega}\cap\{ (x,y)\; : \; x\geq x_{0}\}$. Since $\varphi(y)>0$, we must have $u_{x}<0$ in $\overline{\Omega}\cap\{ (x,y)\; : \; x\geq x_{0}\}$ as well.  

If $\lambda>1$, then $v=u_{x}$ still solves \eqref{quasilinearize}. For $\epsilon_{0}$ sufficiently small, $-b_{z}(u,\lambda)\leq 0$, by \eqref{b}. As before, the strong maximum principle and boundary conditions now yield the desired conclusion. 
\end{proof}

\begin{remark}\thlabel{asyp_ext}
The above lemma is stated for the half strip $(x_{0},\infty) \times (-\frac{\pi}{2},\frac{\pi}{2})$, for some $x_{0}>0$, but a similar result holds for sets of the form $(x_{0}, \infty)\times (0,\frac{\pi}{2})$ or $(-\infty,-x_{0})\times (0,\frac{\pi}{2})$.  
\end{remark}

Next, we consider the nodal properties of a monotone solution. 

\begin{lemma}[Nodal properties]
\thlabel{nodalproplemma}
Let $(u,\lambda) \in \mathcal{O}_{\delta} \cap \mathcal{F}^{-1}(0),\;\;\text{for some} \; \;\delta>0$. Suppose that $u_{x}<0$ in $\Omega^{+}$, $u_{y}<0$ in $\Omega_{+}$, and $u \in X$. Then $u$ satisfies \eqref{nodalprop}.  \end{lemma}

\begin{proof}
Note $u_{x}=0$ on $\partial \Omega^{+}$ from the boundary conditions and evenness in the $x$ variable. In particular, $u_{x}=u_{xx}=0$ on $T$. The Hopf lemma (\hyperref[hopf]{\thref{max_pricp}.\ref{hopf}}) shows that $u_{xx}<0$ on $L$, $u_{xy}<0$ on $B$, and that $u_{xy}>0$ on $T$. Moreover, $u_{xy}=u_{xyy}=0$ on $L$, since $u_{x}=0$ on $L$. If $(s_{1},s_{2})$ is a unit outward pointing vector at $(0,\frac{\pi}{2})$ with $s_{1}<0$ and $s_{2}>0$, then Serrin's lemma (\hyperref[serrin]{\thref{max_pricp}.\ref{serrin}}) requires $\partial^{2}_{s}u_{x}<0$ at $(0,\frac{\pi}{2})$ since $u_{xx}=u_{xy}=0$ at $(0,\frac{\pi}{2})$. A simple calculation shows
$\partial_{s}^{2}u_{x}=s_{1}^{2}u_{xxx}+2s_{1}s_{2}u_{xxy}+s_{2}^{2}u_{xyy}<0$ at $(0,\frac{\pi}{2})$. From this we see $u_{xxy}>0$ at $(0,\frac{\pi}{2})$. A similar argument shows $u_{xxy}<0$ at $(0,-\frac{\pi}{2})$. We are left only to show that $u_{yy}<0$ on $M$. The evenness of $u$ in the $y$ variable implies that $u_{y}=0$ along $M$, and the result follows from the Hopf lemma. 
\end{proof}

We now show that the collection of $(u,\lambda)$ satisfying \eqref{nodalprop} is both open and closed in an appropriate relative topology. 
\begin{lemma}[Open property]
\thlabel{open}
Let $(u,\lambda), (\Tilde{u},\Tilde{\lambda}) \in \mathcal{O}_{\delta} \cap \mathcal{F}^{-1}(0),\;\;\text{for some} \; \;\delta>0$.  Suppose that $0< \lambda, \Tilde{\lambda}$ and $u, \Tilde{u} \in C_{\textup{b,e}}^{3}(\overline{\Omega})\cap C_{0}^{2}(\overline{\Omega})$. If $u$ satisfies \eqref{nodalprop}, then there is some $\epsilon_{0}>0$ for which $|u-\Tilde{u}|_{3}+|\lambda -\Tilde{\lambda}|<\epsilon_{0}$ implies $\Tilde{u}$ also satisfies \eqref{nodalprop}. 
\end{lemma}
\begin{proof}
We will establish the sign of either $\Tilde{u}_{x}$ or $\Tilde{u}_{y}$ in several finite regions, and then invoke \thref{asymptotic} to determine the signs in a leftover tail region. See Figure~\ref{domains} for a sketch of the domains used. Now, because $u \in C_{0}^{2}(\overline{\Omega})$, there is an $R>0$ large enough so that $|u|_{2}< \epsilon/2$ for $x>R$, where $\epsilon$ is chosen to satisfy \thref{asymptotic}. If $|u-\Tilde{u}|_{2}<\epsilon_{1}=\epsilon/2$, then $|\Tilde{u}(x,y)|_{2}<\epsilon$ for $x>R$. Let $\Omega^{+,2R}$ be the rectangle $(0,2R)\times(-\frac{\pi}{2},\frac{\pi}{2})$ and $\Omega^{+}_{k}$ the inscribed rectangle with distance $1/k$ from $\Omega^{+,2R}$. Let us define several regions useful for our analysis: 
\begin{gather*}
T_{k}:=\{(x,\frac{\pi}{2})\; : \; 1/k<x<2R-1/k\} \\ B_{k};=\{(x,- \frac{\pi}{2})\; : \; 1/k<x<2R-1/k\} \\ L_{k} := \{(0,y) \;: \: -\frac{\pi}{2}+1/k<y< \frac{\pi}{2} - 1/k \}. 
\end{gather*}
For a given $k>0$, there is an $\epsilon_{k}$ such that $|u-\Tilde{u}|_{3}<\epsilon_{k}$ implies $\Tilde{u}_{x}<0$ in $\overline{\Omega}_{k}$, $\Tilde{u}_{xx}>0$ on $L_{k}$, $\Tilde{u}_{xy}>0$ on $T_{k}$, and $\Tilde{w}_{xy}<0$ on $B_{k}$. 

Suppose $\epsilon_{0}<1$, and consider the Taylor expansion of $\Tilde{u}_{x}$ at a point $(x_{0},\frac{\pi}{2})$ on $T_{k}$:
\begin{gather} \label{tay}
    \Tilde{u}_{x}(x_{0},y)=\Tilde{u}_{xy}(x_{0},\frac{\pi}{2})(y-\frac{\pi}{2})+O((y-\frac{\pi}{2})^{2}) \; \;\; \text{in} \;\; \;C^{0}(\overline{\Omega}),
\end{gather}
where $\frac{\pi}{2}-1/k < y <\frac{\pi}{2}$. When $k$ is large enough, the remainder term in \eqref{tay} is dominated by the first term and $\Tilde{u}_{x}(x_{0},y)<0$. Analogous arguments show that for large enough $k$, $u_{x}<0$ in the rectangle $(0,1/k)\times (-\frac{\pi}{2}+1/k,\frac{\pi}{2} - 1/k)$, and that $u_{x}<0$ in $(1/k, 2R-1/k)\times(0,1/k)$. 

We still need to deal with the corners. For a given $k$, consider the quarter circle of radius $\frac{\sqrt{2}}{k}$ in $\Omega^{+,2R}$ centered at $(0,\frac{\pi}{2})$. Because $u_{x}, u_{xx}, u_{xy}, u_{xxx} =0$ at $(0,\frac{\pi}{2}$), 
\begin{gather*}
    \Tilde{u}_{x}(x,y)=\Tilde{u}_{xxy}(0,\frac{\pi}{2})(x)(y-\frac{\pi}{2})+O((y-\frac{\pi}{2})^{2})\;\;\; \text{in} \;\;\;C^{0}(\overline{\Omega}).
\end{gather*}
For a given $k$ there exists an $\epsilon_{k}'$ so small that $|u-\Tilde{u}|_{3}<\epsilon_{k}'$ implies that $\Tilde{u}_{xxy}(0,\frac{\pi}{2})>0$. Arguing like before, we see that $\Tilde{u}_{x}(x,y)<0$ in the quarter circle, whenever $k$ is sufficiently large. A similar argument shows that $u_{x}<0$ in quarter circle of radius $k$ centered at $(0,-\frac{\pi}{2})$. 
\begin{figure} \label{domains}
\centering
\begin{tikzpicture}[scale=0.62]
\draw  (0,.43) rectangle (10,3.57) ; 
\draw[thick] (8,0.43)--(8, 3.57); 
\draw[draw] (0, .43) --(1, 0.43) -- (1,.43) arc [start angle=0, end angle=90, radius=1]--(0, .43);
\draw[draw] (0,3.57) --(0, 2.57) -- (0,2.57) arc [start angle=-90, end angle=0, radius=1]--(0, 3.57);
\draw[dashed] (0.75, 0.43)--(0.75, 3.57); 
\draw[dashed] (7.25,0.43)--(7.25,3.57); 
\draw [dashed] (0,1.02)--(8, 1.02); 
\draw [dashed] (0, 2.97)--(8, 2.97); 
\node[align=left] at (-0.2,4.2) {$(0,\frac{\pi}{2})$};
\node[ align=left] at (-0.2,-.3) {$(0,-\frac{\pi}{2})$};
\filldraw [black] (0,0.43) circle (2pt);
\filldraw[black] (0,3.57) circle (2pt); 
\node[align=left] at (4,-0.2) {$(R,0)$}; 
\node[align = left] at (8,-0.2){$(2R,0)$};
\filldraw [black] (4,0.43) circle (2pt);
\filldraw[black] (8,0.43) circle (2pt);
\node[align=left] at (4,4) {$T_{k}$}; 
\node[ align=left] at (4,2) {\large $\Omega^{+}_{k}$};
\node[align=left] at (-0.45, 2) {$L_{k}$};
\end{tikzpicture}
\hspace{0.5cm}
\begin{tikzpicture}[scale=0.62]
\draw [black] (0, .43) rectangle (12,3.57); 
\draw[black, dashed] (2,1)--(10,1);
\node[ align=left] at (6,2) {\large $\Omega_{+,k}$};
\node[ align=left] at (6,4) {$(0,\frac{\pi}{2})$};
\node[ align=left] at (6,-0.2) {$(0,0)$};
\filldraw [black] (6,0.43) circle (2.3pt);
\filldraw[black] (6, 3.57) circle (2.3pt);
\draw (10, 0.43)--(10,3.57); 
\draw(2, 0.43)--(2,3.57); 
\node[ align=left] at (8,-0.2) { $M^{2R}$};
\node[align =left] at (10, -0.2) {$(2R,0)$}; 
\node[align=left] at (2,-0.2) {$(-2R,0)$}; 
\filldraw [black] (10,0.43) circle (2pt);
\filldraw[black] (2,0.43) circle (2pt); 
\end{tikzpicture}
\caption{Left: Regions use to control the sign of $\tilde{u}_{x}$. Right: Regions used to control the sign of $\tilde{u}_{y}$. }
\end{figure}
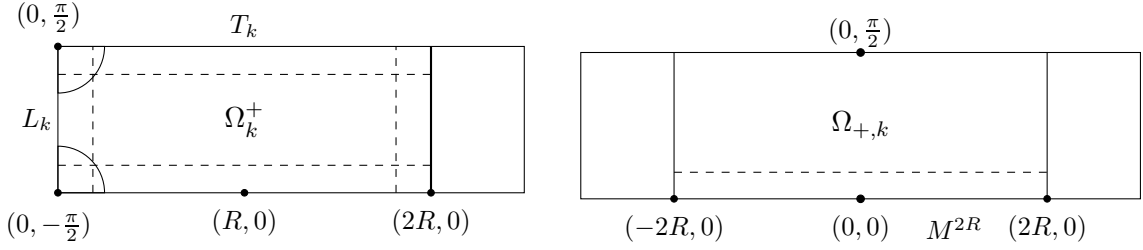
From the work above, we find that if $k$ is taken sufficiently large, and $\epsilon$ taken sufficiently small, then $\tilde{u}_{x}<0$ in $\Omega^{+,2R}$. In particular, $\tilde{u}_{x}<0$ on the line segment with $x=R$ and $-\frac{\pi}{2} < y < \frac{\pi}{2}$. \thref{asymptotic} now implies that $\tilde{u}_{x} <0$ in $\Omega^{+}$. If we can show that $\tilde{u}_{y}<0$ in $\Omega_{+}$, then we will be able invoke \thref{nodalproplemma} to get the desired result. 

The argument to establish a sign on $\tilde{u}_{y}$ is similar to the one just given for $\tilde{u}_{x}$, so we provide only a sketch. Let $\Omega_{+,2R} = \Omega_{+}\cap \{ (x,y) \; : \; |x |\leq 2R\}$ and $\Omega_{+,k} = \Omega_{+}\cap \{ (x,y) \; : \; |x| \leq 2R, y>\frac{1}{k}\}$ and $M^{2R} = M \cap \{(x,y) \; : \; |x| \leq 2R\}$. We see that \eqref{quasilinearize} holds for $v=u_{y}$ in $\Omega_{+}$, except the homogeneous Dirichlet condition is lost. Thus, $u_{y}$ satisfies a uniformly elliptic PDE with a non positive zeroth order coefficient in the tail region $\Omega_{+}\cap \{ (x,y) \; : \; |x| \geq R\}$, where $R$ is the same constant from the above argument. For small enough $\epsilon_{0}$ and $k$ we find that $\tilde{u}_{y}<0$ on $\Omega_{+,k}$ and $\tilde{u}_{yy}  < 0$ on $M^{2R}$. If $k$ is sufficiently large, then from a Taylor expansion along $M^{2R}$ we find that $\tilde{u}_{y}<0$ in $\Omega_{+,2R}$. Thus, a sign condition for $\tilde{u}_{y}$ is established in the finite region $\Omega_{+,2R}$. To deal the corresponding infinite tails, we just need to establish good boundary values, since we know $\tilde{u}_{y}$ satisfies the maximum principle whenever $|x|>R$ (see \thref{asyp_ext}). From the above argument, we have seen that if $\epsilon_{0}$ is small enough, then $\tilde{u}_{xy}>0$ on $T$. This, along with the decay of $\tilde{u}$ at $x = \infty$, is enough to establish that $\tilde{u}_{y}<0$ on all of $T$. A symmetric argument will show that $u_{y}<0$ on the line segment $\{(x,\frac{\pi}{2})\; : \; -\infty< x <0 \}$. Also, $\tilde{u}_{y}=0$ on $M$ by evenness in the $y$ variable. Finally, since $\tilde{u}_{y}\leq0$ on the line segment with $x=R$ and $0\leq y \leq \frac{\pi}{2}$ (and on the segment with $x=-R$ and $0 \leq y \leq \frac{\pi}{2}$, by evenness in the $x$ variable), we conclude that $\Tilde{u}_{y}<0$ on $\Omega_{+}$.    
\end{proof}

\begin{lemma}[Closed property] \thlabel{closedprop}
Let $\{(u_{n},\lambda_{n})\} \subset \mathcal{O}_{\delta}\cap \,\mathcal{F}^{-1}(0)$, for some $\delta>0$. Suppose that $(u_{n},\lambda_{n})\to (u,\lambda)$ in $C^{3}_{\textup{b}}(\overline{\Omega})\times \mathbb{R}$. If each $u_{n}$ satisfies \eqref{nodalprop}, then so does $u$, unless $u\equiv0$. 
\end{lemma}
\begin{proof}
By continuity we have that $u_{x}\leq 0$ in $\Omega^{+}$, $u_{x}=0$ on $\partial \Omega$, and $u_{y} \leq 0$ in $\Omega_{+}$. So $u_{x}$ and $u_{y}$ each satisfy the strong maximum principle (\hyperref[strong_max]{\thref{max_pricp}.\ref{strong_max}}) in the relevant domain because $\mathcal{F}_{u}(u,\lambda)u_{x}=0$ and $\mathcal{F}_{u}(u,\lambda)u_{y}$. Hence, if $u_{x}$ is not trivial, then $u_{x}<0$ in $\Omega^{+}$ and $u_{y}<0$ in $\Omega_{+}$. \thref{nodalproplemma} now implies that $u$ satisfies \eqref{nodalprop}.    
\end{proof}

Next, we show that \eqref{nodalprop} holds along $\mathcal{C}^{I,II}_{\text{loc}}$, which in turn shows that they hold on all of $\mathcal{C}^{I,II}$. 

\begin{lemma}[Nodal properties of the local curve] \thlabel{npropsmall}
 If $(u^{\epsilon}, \epsilon^{2}) \in \mathcal{C}^{I,II}_{\textup{loc}}$ and $0< \epsilon \ll 1$, then $u^{\epsilon}$ exhibits the nodal properties \eqref{nodalprop}. 
\end{lemma}
\begin{proof}
In \thref{Positive}, we established that $u^{\epsilon}_{x}<0$ in $\Omega^{+}$. Since $u^{\epsilon}_{x} =0$ on $T$, the Hopf lemma (\hyperref[hopf]{\thref{max_pricp}.\ref{hopf}}) implies that $u^{\epsilon}_{xy}>0$ on $T$. From \eqref{smallsol}, we know that $u^{\epsilon}_{y}(0,\frac{\pi}{2})<0$ for small enough $\epsilon$. Combining this with the decay of $u_{y}^{\epsilon}$ at infinity allows us to conclude that $u^{\epsilon}_{y}<0$ along all of $T$.

We now proceed as in the proof of \thref{Positive}. We have seen that $u^{\epsilon}$ satisfies equation \eqref{quasilinearize}. If we consider the corresponding uniformly elliptic operator acting on $x$-independent functions of the form $v=f(y)\cos(\sqrt{1-\lambda}y)$, then we obtain an expression with the following asymptotics in $C^{0}(\overline{\Omega})$ 
\begin{empheq}{align}
\label{zz}
\begin{split}
&(1+O(\epsilon^{2})(f''\cos(\sqrt{1-\lambda}y)-2\sqrt{1-\lambda}f'\sin(\sqrt{1-\lambda}y)) \\+ O(\epsilon^{2})(f'&\cos(\sqrt{1-\lambda}y)-\sqrt{1-\lambda}f\sin(\sqrt{1-\lambda}y))+O(\epsilon^{2}f\cos(\sqrt{1-\lambda}y)).
\end{split}
\end{empheq}
Inspecting \eqref{zz} shows that if we choose $f=\Phi_{\epsilon}$, as in the proof of \thref{Positive}, then for sufficiently small $\epsilon$ we can ensure \eqref{zz} is negative for $0\leq y \leq \frac{\pi}{2}$. The boundary condition on $T$, evenness in $y$ (which implies $u^{\epsilon}_{y}=0$ on $M$), maximum principle for uniformly elliptic operators with a positive super-solution (\hyperref[pos_sup_sol]{\thref{max_pricp}.\ref{pos_sup_sol}}), and decay in $x$ are now enough to conclude that $u^{\epsilon}_{y}<0$ in $\Omega_{+}$. Now, \thref{nodalproplemma} implies the result. 
\end{proof}

\begin{theorem}[Global nodal properties]
\thlabel{globalnodal}
Every $(u,\lambda) \in \mathcal{C}^{I,II}$ exhibits the nodal properties \eqref{nodalprop}. 
\end{theorem}
\begin{proof}
Let $(u,\lambda) \in \mathcal{C}^{I,II}_{\text{loc}}$. From \thref{npropsmall} $u$ satisfies \eqref{nodalprop}. Since the nodal properties are both open and closed in the relative topology of $\mathcal{C}^{I,II}$ by \thref{open} and \thref{closedprop}, we conclude that they hold everywhere on $\mathcal{C}^{I,II}$. 
\end{proof}

\end{section}

\begin{section}{Uniform regularity and bounds on loading parameter}
\label{UR}
The main result of this section, which is stated in \thref{apriori}, is that $|u(s)|_{3+\alpha}$ is uniformly bounded along $C^{I,II}_{\delta}$. This is achieved by first using Schauder theory to estimate $|u(s)|_{3+\alpha}$ in terms of $|\nabla u(s)|_{0}$ and then estimating $|\nabla u(s)|_{0}$ in terms of $|u(s)|_{0}$ and $|\lambda(s)|$ by a maximum principle argument. Upper bounds on $|u(s)|_{0}$ and $|\lambda(s)|$ are then established for $C_{\delta}^{I,II}$. Finally, for $s\gg1$ it is shown that there is a positive uniform lower bound on $\lambda(s)$ along $\mathcal{C}^{I,II}_{\delta}$.  

\subsection{A conserved quantity and $L^{p}$ estimates}
We derive a conserved quantity of the system that will play a key role establishing uniform bounds on $|u(s)|_{0}$ and $|\lambda(s)|$ along $\mathcal{C}^{I,II}_{\delta}$. These results, in tandem with \thref{sperbtype}, give our desired a priori estimates. The following calculation is valid for any $C^2$ solution of \eqref{cub}. Let 
\begin{gather}
    \mathcal{L}(z,\xi,\eta,\lambda):=\frac{1}{2}\mathcal{W}(|\xi^2+\eta^2|^{2})+B(z,\lambda)
\end{gather}  
where $$B(z,\lambda) := \int_{0}^{z}b(t, \lambda)\, dt.$$ The anti-plane elastostatic problem \eqref{cub} is formally the Euler--Lagrange equation given formally by 
\begin{gather*}
    \delta \, \int_{\Omega} \mathcal{L}(u,|\nabla u|^{2},\lambda) \, dx \, dy = 0. 
\end{gather*}
Naturally, the translation invariance in $x$ of our system leads us to expect a corresponding conserved quantity. Consider the functional 
\begin{empheq}{align}
\label{hamiltonian}
\begin{split}
   \mathcal{H}(u,\lambda; x):&= \int_{-\frac{\pi}{2}}^{\frac{\pi}{2}} (\mathcal{L}(u,|\nabla u|^{2},\lambda)-\mathcal{L}_{\xi}(u,|\nabla u|^{2}, \lambda) u_{x}) \, dy  \\
   &= \int_{-\frac{\pi}{2}}^{\frac{\pi}{2}} (\frac{1}{2}\mathcal{W}(|\nabla u|^{2})-\mathcal{W}'(|\nabla u|^{2}) u_{x}^{2} + B(u,\lambda) )\, dy.
\end{split}
\end{empheq}
If $(u,\lambda)$ solves \eqref{cub}, then $\mathcal{H}(u,\lambda; \cdot)$ is constant in $x$:
\begin{align*}
    \partial_{x}\mathcal{H}&= \int_{-\frac{\pi}{2}}^{\frac{\pi}{2}}\left( \mathcal{L}_{z}u_{x}+\mathcal{L}_{\xi}u_{xx}+\mathcal{L}_{\eta}u_{xy}-(\partial_{x}\mathcal{L}_{\xi})u_{x}-\mathcal{L}_{\xi}u_{xx}\right) \,dy  \\ 
    &=\int_{-\frac{\pi}{2}}^{\frac{\pi}{2}} u_{x}(\mathcal{L}_{z}-\partial_{y}\mathcal{L}_{\eta}-\partial_{x}\mathcal{L}_{\xi} )\,dy =0
\end{align*}
where we used integration by parts and that
\begin{gather*}
    (\mathcal{L}_{z}-\partial_{y}\mathcal{L}_{\eta}-\partial_{x}\mathcal{L}_{\xi})(y,u,u_{x},u_{y},\lambda) = \mathcal{F}(u,\lambda)=0.
\end{gather*}
It is clear that $\mathcal{H}(u(x,y),\lambda; x)\to0$ as $x \to \infty$, so $\mathcal{H}$ is identically $0$. We record this as a lemma. Note that the arguments of $\mathcal{H}$ will often be suppressed in the sequel. 

\begin{lemma}[Conserved quantity]
\thlabel{Hconst}
Let $u \in C^{2}(\overline{\Omega})$ be a solution to \eqref{cub} for a fixed $\lambda$. Then $\mathcal{H}$ is constant in $x$. In particular, if $u \in X$, then $\mathcal{H}=0$.  
\end{lemma}

The conserved quantity and growth conditions of both $\mathcal{W}$ and $b$ are enough obtain a uniform bound on $|u(s)|_{0}$ (and on $|\lambda(s)|$, as shown in Subsection~\ref{boundsonlambda_sub}). 

\begin{lemma}[$L^{p}$ bounds]
\thlabel{top}
There exists a constant $C(c_{1}, c_{2}, b_{1})$ such that if $u \in X$ is a solution to \eqref{cub}, corresponding to Model I, with $0<\lambda $, then 
\begin{gather}
    \|u(0,\cdot)\|_{2},\;\; \|u_{y}(0,\cdot)\|_{6}, \; \; |u(0,\cdot)|_{1/2} \leq C,
\end{gather}
where $\| \cdot \|_{p} $ denotes the $L^{p}$ norm on $[-\frac{\pi}{2},\frac{\pi}{2}]$. Moreover, for any $x_{0} \in \mathbb{R}$ we have
\begin{gather*}
    \|u(x_{0},\cdot)\|_{2}, \;\;|u(x_{0},0)|_{0} \leq C
\end{gather*}
\end{lemma}
\begin{proof}
From \eqref{hamiltonian}, \thref{Hconst}, \eqref{wgrowth2} and \eqref{bcond} we see that when $x=0$
\begin{align*}
    0=2\mathcal{H} = 2\int_{-\frac{\pi}{2}}^{\frac{\pi}{2}}\mathcal{L}\, dy \geq& \|u_{y}(0,\cdot)\|_{2}^{2}+c_{1}\|u_{y}(0,\cdot)\|_{4}^{4}+c_{2}\|u_{y}(0,\cdot)\|_{6}^{6}\\+&(\lambda- 1)\|u(0,\cdot)\|_{2}^{2}+\frac{b_{1}}{2}\|u(0,\cdot)\|_{4}^{4},
\end{align*}
For the remainder of the proof, we will suppress arguments of $u(0,\cdot)$ and $u_{y}(0,\cdot)$ appearing in $L^{p}[-\frac{\pi}{2},\frac{\pi}{2}]$ norms. Wirtinger's inequality implies that $\|u_{y}\|_{2}^{2}-\|u\|_{2}^{2} \geq 0$, and $\|u\|_{4} \leq \pi\|u_{y}\|_{4}$ by Friedrichs's inequality, so that 
\begin{gather}
\label{in1}
    c_{1}\|u_{y}\|_{4}^{4}+c_{2}\|u_{y}\|_{6}^{6}+\frac{b_{1}\pi}{2}\|u_{y}\|_{4}^{4}\leq -\lambda\|u\|_{2}^{2} +(\|u\|_{2}^{2}-\|u_{y}\|_{2}^{2})\leq 0.
\end{gather}
H\"older's inequality yields 
\begin{gather*}
\label{in2}
    \|u_{y}\|^{4}_{4} \leq \pi^{\frac{1}{3}}\|u_{y}\|_{6}^{4}.
\end{gather*}
Altogether these give 
\begin{gather} \label{upperlambda}
    (c_{1}+\frac{b_{1}\pi}{2})\pi^{\frac{1}{3}}\|u_{y}\|_{6}^{4}+c_{2}\|u_{y}\|_{6}^{6} \leq 0 
\end{gather}
so that 
\begin{gather*}
    \|u_{y}\|^{2}_{6} \leq \dfrac{|c_{1}+\frac{b_{1}}{2}\pi|\pi^{\frac{1}{3}}}{c_{2}}.
\end{gather*}
Thus, $\|u_{y}\|_{6}$ is uniformly bounded. An application of H\"older's inequality shows that $\|u_{y}\|_{2}$ is uniformly bounded too. As mentioned above, $\|u\|_{2} \leq \|u_{y}\|_{2}$. Hence, we obtain a uniform bound on $|u(0,\cdot)|_{1/2}$ by Sobolev embedding. Because of the monotonicity of $u$ established in \eqref{nodalprop}, we see that $\|u(x_{0},\cdot)\|_{2}$ and $|u(x_{0},\cdot)|_{0}$ are maximized at $x_{0}=0$. Thus, the $L^{2}$ and $L^{\infty}$ norms of $u$ are uniformly bounded on any transversal line in $\Omega$.   
\end{proof}

\begin{remark}\thlabel{bounds_u_II}
Note that this says nothing about solutions of Model II. If $u \in \mathcal{C}^{II}_{\delta}$, with $\delta>0$, then $|u|<C$, where $C$ depends on $q_{1}$. This is a direct consequence of \eqref{wnondegen}, \eqref{wdegen}, and the homogeneous Dirichlet conditions ($|\nabla u|^{2}<q_{2}$ along $\mathcal{C}_{\delta}^{II}$).
\end{remark} 

\begin{remark}
More generally, if $\mathcal{H}=M$, then the above argument shows that $\|u_{y}(0,\cdot)\|_{6}$ is bounded uniformly by a constant $C$ that depends on $c_{1},c_{2},b_{1}$ and $M$, so long as one assumes sufficient growth of $\mathcal{W}$ relative to $b$.
\end{remark}
At this stage, we are left to establish control on $|\lambda(s)|$ to complete our desired estimates of $|u(s)|_{3+\alpha}$ 

\subsection{Uniform regularity}
We begin by using the so called ``$P$-function" technique (see \cite{sperb}) along with standard elliptic estimates to gain some control on $|u(s)|_{3+\alpha}$.  
\begin{lemma}
\thlabel{sperbtype}
Let $(u,\lambda) \in \mathcal{O}_{\delta}\cap \,\mathcal{F}^{-1}(0)$, for some $\delta>0$. If $\lambda$ and $K$ are positive, and $|u|_{0}+\lambda<K$, then there is a constant $C(K,\delta)>0$ for which $|u|_{3+\alpha} \leq C(K,\delta)$. If $(u,\lambda)$ is a solution which corresponds to Model I, then the above estimate holds for some $C = C(K)$. 
\end{lemma}
\begin{proof}
We prove this result by using a maximum principle of Payne and Philippin. First, we obtain bounds on $|\nabla u(s) |^{2}$. Recall, as mentioned in \thref{bounds_u_II} that this is trivial for Model II, so let us assume for now the conditions of Model I. By Theorem 1 of \cite{Payne} the function 
\begin{gather*}
    P(x,y)=\int_0^{|\nabla u(x,y)|^{2}}(\mathcal{W}'(\xi)+2\xi\mathcal{W}''(\xi))d\xi - 2\int_{0}^{u(x,y)}b(\eta,\lambda)\eta \,d\eta,
\end{gather*}
obtains its maximum either on $\partial \Omega$ or at a critical point of $u$. We should note that in \cite{Payne} the results are stated for bounded $C^{2+\alpha}$ domains. So, our application includes the additional possibility that the maximum of $P$ occurs in the limit as $x \to \pm \infty$. However, the decay of $u$ precludes this scenario for nontrivial solutions. The homogeneous Dirichlet boundary conditions of \eqref{cub} and monotonicity properties of \eqref{nodalprop} now imply that $P$ is maximized at $(0,0)$, which is the only critical point of $u$. Thus,  
\begin{gather*}
    (2q\mathcal{W}'(q)-\mathcal{W}(q))\big\vert_{q=|\nabla u(x,y)|^{2}}-2\int_{0}^{u(x,y)}b(\eta,\lambda)\eta \,d\eta \leq -2\int_{0}^{u(0,0)}b(\eta,\lambda)\eta \, d\eta.
\end{gather*}
So, 
\begin{gather} \label{gradest}
    2q\mathcal{W}'(q)-\mathcal{W}(q) \leq -2 \int_{u(x,y)}^{u(0,0)}b(\eta,\lambda)\eta \; d\eta \leq 2(u(0,0))^{2} \max\limits_{(x,y) \in \overline{\Omega}}{|b(u(x,y),\lambda)|}.
\end{gather}
Since $b$ is analytic in both $z$ and $\lambda$, it follows that the right hand side of \eqref{gradest} is bounded by $C(K)$. Moreover, the left hand side of \eqref{gradest} satisfies 
\begin{gather*}
    2q\mathcal{W}'(q)-\mathcal{W}(q) = \int_{0}^{q}\left(\mathcal{W}'+2q\mathcal{W}''(q)\right)\;dq \geq q\xi_{1},
\end{gather*}
whenever $q\geq 0$, by \eqref{Wcond}. Hence, 
\begin{gather}\label{grad_sup}
    |\nabla u|_{0}^{2} \leq \dfrac{2(u(0,0))^{2}}{\xi_{1}} \max\limits_{(x,y) \in \overline{\Omega}}{|b(u(x,y),\lambda)|}.
\end{gather}

Standard elliptic theory can now be invoked to upgrade a uniform bound in $|\nabla u|$ into a uniform bound in $C^{3+\alpha}(\overline{\Omega})$. As we have seen in \eqref{quasilinearize}, $\partial_{x}u$ solves a divergence form elliptic equation. In particular, from \eqref{gradest} it follows that we may view $\partial_{x}u$ as the solution to a linear PDE with uniformly bounded coefficients. An application of \cite[Theorem 8.29]{gilbarg2015elliptic} yields that for some $\alpha' \in (0,\alpha] $
\begin{gather*}
    |u_{x}|_{C^{\alpha'}(\Omega_{M})} \leq C 
\end{gather*}
where $\Omega_{M} = \Omega\cap \{(x,y) \; : \; M\leq x \leq M+1$\}, and both $\alpha'$ and $C$ depend on $K$ and $\delta$ (or only $K$ in the case of Model I). An analogous bound for $|u_{y}|_{C^{\alpha'}(\Omega_{M})}$ is obtained by differentiating in $y$ instead. Now, by viewing \eqref{cub} as a linear equation with coefficients that depend on $u_{x}$ and $u_{y}$, which we have just shown are uniformly bounded in $C^{\alpha'}(\Omega_{M})$, we may apply (linear) Schauder theory to obtain a uniform bound on $|u|_{C^{2+\alpha'}(\Omega_{M})}$. This gives control over $|u|_{C^{1+\alpha}(\Omega_{M})}$, so that by repeating the previous argument we gain control of $|u|_{C^{2+\alpha}(\Omega_{M})}$. Now, Schauder estimates applied to the linearized equations for either $u_{x}$ or $u_{y}$ provide a uniform bound on $|u|_{C^{3+\alpha}(\Omega_{M})}$. 
\end{proof}

\subsection{Bounds on loading parameter}
\label{boundsonlambda_sub}
Now, we show that as we follow either $\mathcal{C}^{I}$ or $\mathcal{C}_{\delta}^{II}$, for some $\delta>0$, that $\lambda(s)$ cannot return to $0$ without the corresponding solutions returning to the reference configuration or the equation undergoing a loss of ellipticity. Moreover, an upper bound on $\lambda(s)$ is derived for either case. These estimates will be used to establish bounds on $|u(s)|_{3+\alpha}$ and \eqref{blowup}. 

\begin{lemma}
\thlabel{boundsL}
If $\{(u_{n}, \lambda_{n})\} \subset \mathcal{C}^{I}$, or $\{(u_{n}, \lambda_{n})\} \subset\mathcal{C}^{II}_{\delta}$, for some $\delta>0$, is sequence of solutions to \eqref{cub}, that are uniformly bounded in $C_{\textup{b}}^{3+\alpha}(\overline{\Omega})$ and for which $\lambda_{n} \to 0$, then $u_{n}\to 0$ in $X$.  
\end{lemma}
\begin{proof}
Assume that $\{u_{n}\}$ does not converge to $0$ in $C_{\text{b}}^{3+\alpha}(\overline{\Omega})$. By the hypothesis and \eqref{nodalprop}, we may then invoke \thref{CorF}. From this, one may conclude that either there is a subsequence converging in $C^{3+\alpha}_{\text{b}}(\overline{\Omega})$ to a solution $(u,0) \in X\times \mathbb{R}$ of \eqref{cub}, or there is a subsequence  of translates
\begin{gather*}
    u_{n}(\cdot +x_{n}, \cdot) \xrightarrow{C^{2}_{\text{b}}(\overline{\Omega})}\tilde{u}(\cdot, \cdot)
\end{gather*}
where $x_{n} \to \infty$, and $(\tilde{u},0) \in C^{3+\alpha}_{\text{b}}(\overline{\Omega})\times \mathbb{R}$ solves \eqref{cub}. Moreover, $\tilde{u}_{x},\tilde{u}_{y} \to 0$ uniformly in $x,y$ as $x\to \infty$, and $\tilde{u}_{x} \leq 0 $. From \eqref{nodalprop} we also know that $\tilde{u}_{y} \geq0$ for $y \in [-\frac{\pi}{2},0)$ and $\tilde{u}_{y} \leq 0$ for $y \in (0,\frac{\pi}{2}]$. An application of the strong maximum principle for signed solutions (see \hyperref[strong_max]{\thref{max_pricp}.\ref{strong_max}}) shows that the inequalities for $\tilde{u}_{x}$ and $\tilde{u}_{y}$ must in fact be strict in the interior of $\Omega$. Throughout the rest of the proof, the properties of $u$ and $\tilde{u}$ that concern us are the same, namely that they are monotone bounded solutions to \eqref{cub} that decay as $x \to \infty$. For simplicity, let us only write $u$ in the sequel, with the understanding that the argument applies equally well to $\tilde{u}$. 

Let $0<R_{1}<R_{2}$ and $\varphi(y)=\cos(y)$. Multiplying \eqref{cub} by $\varphi$ and then integrating we find 
\begin{align*}
    0&=\int_{-\frac{\pi}{2}}^{\frac{\pi}{2}}\int_{R_{1}}^{R_{2}}(\nabla \cdot (\mathcal{W}'(|\nabla u|^{2})\nabla u)-b(u,\lambda))\varphi\,dxdy \\
    &=\int_{-\frac{\pi}{2}}^{\frac{\pi}{2}}\mathcal{W}'(|\nabla u|^{2})u_{x}\varphi\big\vert_{R_{1}}^{R_{2}}\;dy+\int_{-\frac{\pi}{2}}^{\frac{\pi}{2}}\int_{R_{1}}^{R_{2}}\left(-\mathcal{W}'(|\nabla u|^2)u_{y}\varphi_{y}-b(u,\lambda)\varphi\right)\,dxdy
\end{align*}
Note that $u_y\varphi_y > 0$ by the comments above. For $R_{1}$ sufficiently large, $$\frac{1}{2} <\mathcal{W}'(|\nabla u|^{2}) < 1 \;\;\;\text{in}\;\;\; (R_{1},\infty)\times(-\frac{\pi}{2},\frac{\pi}{2})$$ because of the decay in $u_{x}$ and $u_{y}$. Recall that $b$ has the form \eqref{b} when $|\nabla u|^{2}$ is sufficiently small. So for large enough $R_{1}$ we also have $-b(u,\lambda)-(1-\lambda)u\geq0$ whenever $x>R_{1}$ and $-\frac{\pi}{2} < y < \frac{\pi}{2}$. Letting $R_{2} \to \infty$ we see that 
\begin{gather*}
\label{contra}
    0=\int_{-\frac{\pi}{2}}^{\frac{\pi}{2}}\int_{R_{1}}^{\infty}(-\mathcal{W}'(|\nabla u|^{2})+1)u_{y}\varphi_{y}+(-b(u,\lambda)-(1-\lambda)u)\varphi \,dxdy \\-\int_{-\frac{\pi}{2}}^{\frac{\pi}{2}}\mathcal{W}'(|\nabla u|^{2})u_{x}\varphi(R_{1},y)\,dy>0.
\end{gather*}
This is of course a contradiction. 
\end{proof} 

\begin{lemma}[Bounds on $\lambda$]
\thlabel{boundsonlambda}
Let $(u,\lambda) \in \mathcal{C}^{I} \setminus\mathcal{C}^{I}_{\textup{loc}}$. Then there exists some positive constants $\lambda_{1}^{\pm}=\lambda^{\pm}(c_{1},c_{2},b_{1})$ for which $0< \lambda^{-} < \lambda<\lambda^{+} < \infty$. If instead we take $(u,\lambda) \in   \mathcal{C}^{II}_{\delta}\setminus\mathcal{C}_{\textup{loc}}$, for some $\delta>0$, then the result still holds with positive constants $\lambda_{2}^{\pm}$, except that $\lambda^{-}$ now depends on $\delta$ as well. 
\end{lemma}
\begin{proof}
First, let us suppose that there is some sequence $\{(u_{n},\lambda_{n})\} \subset \mathcal{C}_{\delta}^{I,II} \setminus\mathcal{C}^{I,II}_{\text{loc}}$ for which $\lambda_{n} \to 0$. By \thref{sperbtype} and \thref{top} we see that $\{u_{n}\}$ is uniformly bounded in $C^{3+\alpha}_{\text{b}}(\overline{\Omega})$ (in the case of Model II there is dependence on $\delta$). Then, from \thref{boundsL}, it follows that $u_{n} \to 0$ in  $C^{3+\alpha}_{\text{b}}(\overline{\Omega})$. However, \thref{Positive} then implies that $u_{n} \in \mathcal{C}^{I,II}_{\text{loc}}$ for large enough $n$, and this contradicts part \ref{notinloc} of \thref{global}. 

To establish an upper bound on $\lambda$, see that \eqref{bcond} and \thref{Hconst} imply 
\begin{gather*}
    0 = \int_{-\frac{\pi}{2}}^{\frac{\pi}{2}}(\mathcal{W}(|\nabla u |^{2})+2B)\big\vert_{x=0}\,dy \geq \int_{-\frac{\pi}{2}}^{\frac{\pi}{2}} \mathcal{W}(|\nabla u|^{2})\big\vert_{x=0}\,dy+(\lambda -1) \|u(0,\cdot)\|_{2}^{2}+\frac{b_{1}}{2}\|u(0,\cdot)\|_{4}^{4},    
\end{gather*}
where we are using $\|\cdot\|_{p}$ to denote the $L^{p}$ norm on $[-\frac{\pi}{2},\frac{\pi}{2}]$. Since $\mathcal{W}(q) > 0$ for $q>0$, it follows that $(\lambda -1)\|u(0,\cdot)\|_{2}^{2}+\frac{b_{1}}{2}\|u(0,\cdot)\|_{4}^{4} \leq 0$. Appealing to \thref{top}, We find that \begin{gather*}
    \lambda \leq \left(1+\frac{|b_{1}|}{2}\sup_{s>0}|u(s)|^{2}\right)\|u(0,\cdot)\|_{2}^{2} \leq C(c_{1}, c_{2}, b_{1}) \qedhere
\end{gather*}
\end{proof}

We are now ready to state the main a priori estimate. 

\begin{proposition} \thlabel{apriori}
If $(u, \lambda) \in \mathcal{C}^{I}$, then then there exists some $C(c_{1},c_{2},b_{1})>0$ for which $|u|_{3+\alpha} \leq C$. If instead $(u,\lambda) \in \mathcal{C}^{II}_{\delta}$, then there exists $C(c_{1},c_{2},b_{1},\delta) >0$ for which $|u|_{3+\alpha} \leq C$. 
\end{proposition}  

\begin{proof}
For $(u,\lambda) \in \mathcal{C}^{I}$, the result follows from \thref{top} and \thref{boundsonlambda} combined with \thref{sperbtype}. For $(u,\lambda) \in \mathcal{C}^{II}_{\delta}$, the estimate can be obtained from \thref{bounds_u_II} and \thref{boundsonlambda} in conjunction with \thref{sperbtype}. 
\end{proof}
\end{section}

\begin{section}{Proof of the main results} \label{main_proofs}
We are now prepared to prove the main results. The existence of a global solution branch, for either model, is shown in Section~\ref{globalsection}. The key difference between the global behavior of $\mathcal{C}^{I}$ and $\mathcal{C}^{II}$ is related to alternative \ref{blowup_alt} of \thref{global}; for $\mathcal{C}^{I}$ it is shown to be impossible, thus forcing alternative \ref{lossofC}, whereas for $\mathcal{C}^{II}$ it is shown to hold (note that \ref{lossofC} is \textit{not} necessarily excluded in this case).  

\begin{proof}[Proof of \thref{thm1}]
By \thref{global}, there exists a curve of solutions $\mathcal{C}^{I}$ extending $\mathcal{C}^{I}_{\text{loc}}$, which is locally real analytic with $C^{0}$ parameterization
\begin{gather*}
    \mathcal{C}^{I} = \{(u(s),\lambda(s)) \; : \; 0<s < \infty \} \subset X\times \mathbb{R}. 
\end{gather*} 
The symmetry and monotonicity properties of \hyperref[sm]{\thref{thm1}.\ref{sm}} are proved throughout in \thref{globalnodal}. The bounds on $\lambda(s)$ and $\sup_{s \geq 0}|u(s)|_{3+\alpha}$ from \hyperref[th1boundslambda]{\thref{thm1}.\ref{th1boundslambda}} and \hyperref[bounds_displacement]{\thref{thm1}.\ref{bounds_displacement}} are established in \thref{boundsonlambda} and \thref{apriori}, respectively. This establishes all parts of \thref{thm1} except for the broadening in \ref{broadening}. From the alternatives in \thref{global}, we see that \ref{broadening} must hold if 
\begin{gather*}
    N(s)= |u(s)|_{3+\alpha}+\dfrac{1}{\text{dist}(u(s),\partial \mathcal{O})}+\lambda(s)+\dfrac{1}{\text{dist}(\lambda(s),\partial \mathcal{I})}
\end{gather*}
is bounded uniformly in $s$. The bound on the first term follows  directly from \thref{sperbtype} and those on the third and fourth terms follow from the estimates on $\lambda$ established in \thref{boundsonlambda}. Recall that $\mathcal{O}$ is defined by \eqref{ODef}, so \eqref{Wcond} implies that the second term remains bounded as well. Hence, we have the desired control over $N(s)$, and the result must hold. \end{proof}

We need to prove one more lemma in preparation for \thref{thm2}.

\begin{lemma}[Nonexistence of monotone fronts]
\thlabel{dontlose}
Let $\mathcal{F}$ correspond to Model I. Then, $\mathcal{F}^{-1}(0) \cap \mathcal{O}_{\delta}$, for $\delta>0$, is locally pre-compact in $X$. In particular, alternative \ref{Front} of \thref{CorF} cannot hold for a sequence $\{(u_{n},\lambda_{n})\} \subset \mathcal{C}^{II}_{\delta}$, $\delta>0$.
\end{lemma}
\begin{proof}
From \thref{CorF} we find that if $\mathcal{F}^{-1}(0) \cap \mathcal{O}_{\delta}$ fails to be locally pre-compact, then \hyperref[Front]{\thref{CorF}.\ref{Front}} must hold. Suppose that $\{u_{n}\}$ is a sequence satisfying \hyperref[Front]{\thref{CorF}.\ref{Front}}. Then $\lim_{n \to \infty} u_{n}(\cdot + x_{n},y)=: \tilde{u}(x,y) \in C^{3+\alpha}_{\text{b}}(\overline{\Omega})$ must solve \eqref{cub}. Moreover, $U(y):= \lim_{x\to -\infty} \tilde{u}(x,y)$ satisfies
\begin{gather}
\label{limitingode}
    ((\mathcal{W}'((U_{y})^2)U_{y})_{y}-b(U,\lambda)=0,
\end{gather}
which can be seen by \cite[Lemma~2.3]{chen2020global}. Multiplying \eqref{limitingode} by $U(y)$ and integrating by parts yields 
\begin{gather*}
\label{eq1}
    0 = \int_{-\frac{\pi}{2}}^{\frac{\pi}{2}}(\mathcal{W}'(U_{y}^{2})U_{y}^{2}+b(U,\lambda)U)\; dy.
\end{gather*}
From \thref{Hconst} we know that $\mathcal{H}(u_{n},\lambda; x)=0$, and hence that $\mathcal{H}(U,\lambda; x)=0$. Written explicitly this becomes  
\begin{gather*}
    0=\int_{-\frac{\pi}{2}}^{\frac{\pi}{2}}\left(\frac{1}{2}\mathcal{W}(U^{2}_{y})+B(U,\lambda)\right)\; dy.
\end{gather*}
After combining these equations, we find
\begin{gather}\label{modelII_cont}
    0 = \int_{-\frac{\pi}{2}}^{\frac{\pi}{2}}(\mathcal{W}'(U_{y}^{2})U_{y}^{2}-\mathcal{W}(U^{2}_{y})+b(U,\lambda)U-2B(U,\lambda))\; dy.
\end{gather}
A simple calculation will show
\begin{gather*}
    b(z, \lambda)z -2 B(z, \lambda) < 0 \; \; \; \text{for} \; \; \;  z>0, 
\end{gather*}
where the concavity of $b(z,\lambda)$ and the fact that $b(0, \lambda)=0$ are used. Recall that $q\mathcal{W}'(q)-\mathcal{W}<0$ by \eqref{wdegendamp}. But then the right hand side of \eqref{modelII_cont} is negative, which is a contradiction, hence the result holds.  
\end{proof}

\begin{proof}[Proof of \thref{thm2}]
From \thref{global}, we see there is a curve of solutions $\mathcal{C}^{II}$, extending $\mathcal{C}^{II}_{\text{loc}}$, which is locally real analytic with $C^{0}$ parameterization
\begin{gather*}
    \mathcal{C}^{II} = \{(u(s),\lambda(s)) \; : \; 0<s < \infty \} \subset X \times \mathbb{R}. 
\end{gather*} 
As in the proof of \thref{thm1}, we wish to understand the alternatives in \hyperref[alternatives]{\thref{global}.\ref{alternatives}}. The quantity
\begin{gather*}
   \inf_{\overline{\Omega}}\big(\mathcal{W}'(q)+2q\mathcal{W}''(q)\big)\big\vert_{q=|\nabla u(s)|^{2}}
\end{gather*}
is not bounded below a priori along $\mathcal{C}^{II}$ as it was for $\mathcal{C}^{I}$. This leads us to consider $N(s)$ (see \eqref{blowup} or the proof of \thref{thm1}) on a segment $\mathcal{C}^{II}_{\delta}$ of $\mathcal{C}^{II}$, with $\delta>0$. An estimate of the form $|u(s)|_{3+\alpha} < C(\delta)$ is obtained whenever $(u(s), \lambda(s)) \in \mathcal{C}^{II}_{\delta}$, by \thref{apriori}. 

Now, if we assume $\mathcal{C}^{II} =\mathcal{C}^{II}_{\delta^{*}}$, for some $\delta^{*}>0$, then the first term of $N(s)$ is uniformly bounded along $\mathcal{C}^{II}$ by the paragraph above. Of course, this assumption also implies that the second term is uniformly bounded along $\mathcal{C}^{II}$ by definition. Furthermore, \thref{boundsonlambda} implies that the third and fourth terms are also controlled. 
Thus, there is some $C'(\delta)$ for which $s \gg 1$ implies $|N(s)| \leq C'(\delta)$. Hence, alternative \ref{lossofC} of \thref{global} must hold. However, this contradicts the impossibility of fronts established in \thref{dontlose}. So we must have
\begin{gather*}
    \lim_{s\to \infty}\inf_{\overline{\Omega}}\big(\mathcal{W}'(q)+2q\mathcal{W}''(q)\big)\big\vert_{q=|\nabla u(s)|^{2}} = 0.
\end{gather*}
In particular, we must have $|\nabla u(s)|^{2} \to q_{1}$. Note that our estimates on $|u(s)|_{3+\alpha}$ and $\lambda(s)$ breakdown as $\delta \to 0$. This leaves open the possibility that $\lambda(s)$ approaches either $0$ or $\infty$, or that a blow-up in $C^{3+\alpha}(\overline{\Omega})$ (note that $|u(s)|_{0}$ and $|u_{y}(s)|_{0}$ are indeed bounded, but the elliptic estimates depend on $\delta$) occurs concurrently with the loss of ellipticity. This establishes \hyperref[loss_ellipt]{\thref{thm2}.\ref{loss_ellipt}}

The monotonicity and symmetry properties of \hyperref[sm2]{\thref{thm2}.\ref{sm2}} are prove in \thref{globalnodal}. Finally, the bound on $\lambda(s)$ of \hyperref[th2boundslambda]{\thref{thm2}.\ref{th2boundslambda}} is proved in \thref{boundsonlambda}. 
\end{proof}

\end{section}

\section*{Acknowledgements}
The author was supported in part by the NSF through DMS-1812436. The author also wishes to thank Samuel Walsh for his advise and guidance throughout the writing of this paper. 
\appendix

\begin{section}{Maximum principles}
We recall some variants of the maximum principle which are used repeatedly throughout the paper. In particular, we state the strong maximum principle, including the version which a sign condition on the zeroth order term is replaced with a sign condition on the solution (\cite[Lemma 1]{serrin1971symmetry}), the Hopf boundary lemma, Serrin edge point lemma (see \cite{serrin1971symmetry}), and finally a variant which holds when a positive supersolution is known to exist (see \cite[Section~5 Theorem~10]{protter2012maximum}). 

\begin{theorem} \thlabel{max_pricp}
Let $\Omega$ be a connected, open set (possibly unbounded), and consider the second order operator $L$ given by 
\begin{gather}
    L: =\sum_{i,j = 1}^{n}a_{ij}(x) \partial_{i}\partial_{j} + \sum_{i=1}^{n}b_{i}(x)\partial_{i}+c(x)
\end{gather}
where the coefficients are of class $C_{0}^{0}(\overline{\Omega})$. Suppose that $L$ is uniformly elliptic in the sense that there exists $\lambda>0$ such that 
\begin{gather}
    \sum_{i,j}a_{ij}(x) \xi_{i}\xi_{j} \geq \lambda |\xi|^{2}, \qquad \text{for all} \; \; \xi \in \real^{n}, \; x \in \overline{\Omega}, 
\end{gather}
and that $(a_{ij})$ is symmetric. Let $u \in C^{2}(\Omega) \cap C^{0}(\overline{\Omega})$ be a solution of $Lu=0$ in $\Omega$. 

\begin{enumerate}[label=\rm(\roman*)]
\item \label{strong_max} \textup{(Strong maximum principle)} Suppose that $u$ attains a maximum value on $\overline{\Omega}$ at a point  in the interior of $\Omega$. If $c \leq 0 $ in $\Omega$, or if $\sup_{\overline{\Omega}}u \leq 0 $, then $u$ is a constant function. 
\item\label{hopf} \textup{(Hopf boundary lemma)} Suppose that $u$ attains it maximum value on $\overline{\Omega}$ at a point $x_{0} \in \partial \Omega$ for which there exists an open ball $B \subset{\Omega}$ with $\overline{B} \cap \partial \Omega = \{x_{0}\}$. Assume that either $c \leq 0$ or else $\sup_{B}u = 0$. Then $u$ is a constant function or 
\begin{gather*}
    \nu \cdot \nabla u(x_{0}) > 0, 
\end{gather*}
where $\nu$ is the outward unit normal to $\Omega$ at $x_{0}$. 
\item \label{serrin} \textup{(Serrin edge point lemma)} Let $x_{0} \in \partial \Omega$ be an edge point in the sense that near $x_{0}$ the boundary $\partial \Omega$ consists of two transversally intersection $C^{2}$ hypersurfaces $\{\sigma(x) =0 \}$ and $\{\gamma(x) =0 \}$. Suppose that $\sigma, \gamma < 0$ in $\Omega$, $u \in C^{2}(\overline{\Omega})$, $u>0$ in $\Omega$ and $u(x_{0}) =0$. Assume further that $a_{ij} \in C^{2}$ in a neighborhood of $x_{0}$,
\begin{gather*}
    B(x_{0}) =0, \qquad \text{and} \qquad \partial_{\tau}B(x_{0})=0, 
\end{gather*}
for every differential operator $\partial_{\tau}$ tangential to $\{\sigma=0\}\cap\{\gamma=0\}$ at $x_{0}$. Then for any unit vector $s$ outward from $\Omega$ at $x_{0}$, either 
\begin{gather*}
    \partial_{s}u(x_{0}) < 0, \; \; \text{and} \; \; \partial^{2}_{s}u(x_{0}) <0.
\end{gather*}
\item \label{pos_sup_sol} \textup{(When a positive supersolution exists)} Suppose that there exists $v \in C^{2}(\overline{\Omega})$ such that $v>0$ and $Lv \leq 0$ in $\overline{\Omega}$. Then either $\dfrac{u}{v}$ is constant, or $\dfrac{u}{v}$ cannot achieve a nonnegative maximum in $\Omega$. \label{max_sup_sol} 
\end{enumerate}
\end{theorem}
\begin{remark}
If, in the context of \ref{max_sup_sol} above we suppose that that $u \geq$ on $\Omega$, then the existence of a positive superolution implies that $u>0$ in $\overline{\Omega}$.
\end{remark}

\end{section}

\bibliography{bibliography}
\bibliographystyle{siam}

\end{document}